\documentclass[11pt]{article}

\usepackage[margin=1in]{geometry}
\usepackage{amssymb, color}
\usepackage{hyperref}
\usepackage{graphicx,xcolor}
\usepackage{amsthm,amsfonts,euscript,amsmath,comment,slashed,here,todonotes}
\usepackage[affil-it]{authblk}

\newtheorem{theorem}{Theorem}[section]
\newtheorem{lemma}[theorem]{Lemma}
\newtheorem{proposition}{Proposition}[section]
\newtheorem{corollary}{Corollary}[section]

\theoremstyle{definition}
\newtheorem{definition}[theorem]{Definition}

\theoremstyle{remark}
\newtheorem{remark}[theorem]{Remark}

\numberwithin{equation}{section}

\newcommand\be{\begin{equation}}
\newcommand\ee{\end{equation}}
\newcommand\bea{\begin{eqnarray}}
\newcommand\eea{\end{eqnarray}}
\newcommand\bi{\begin{itemize}}
\newcommand\ei{\end{itemize}}
\newcommand\ben{\begin{enumerate}}
\newcommand\bena{\begin{enumerate}[(a)]}
\newcommand\een{\end{enumerate}}

\newcommand\bp{\begin{proof}}
\newcommand\ep{\end{proof}}

\newcommand{\R}{\ensuremath{\mathbb{R}}}

\newcommand{\N}{\mathbb{N}}

\renewcommand\Re{\operatorname{Re}}

\newcommand\p{\partial}

\newcommand{\eps}{\varepsilon}

\newcommand{\n}{\nonumber}

\allowdisplaybreaks

\title{Local Rigidity of the Couette Flow \\ for the Stationary Triple-Deck Equations}
\author{Sameer Iyer\footnote{Department of Mathematics, University of California, Davis, Davis, CA 95616, USA, \url{sameer@math.ucdavis.edu}} \qquad Yasunori Maekawa\footnote{Department of Mathematics, Graduate School of Science, Kyoto University, Kyoto 606-8502, Japan, \url{maekawa.yasunori.3n@kyoto-u.ac.jp}}}

\begin{document}

\maketitle

\begin{abstract}  The Triple-Deck equations are a classical boundary layer model which describes the asymptotics of a viscous flow near the separation point, and the Couette flow is an exact stationary solution to the Triple-Deck equations. In this paper we prove the local rigidity of the Couette flow in the sense that there are no other stationary solutions near the Couette flow in a scale invariant space. This provides a stark contrast to the well-studied stationary Prandtl counterpart, and in particular offers a first result towards the rigidity question raised by R. E. Meyer in 1983.
\end{abstract}

\section{Introduction}

In this paper we consider the two-dimensional stationary Triple-Deck equations in $\R\times \R_+$ with $\R_+=(0,\infty)$:
\begin{subequations}\label{TD:1}
\begin{align}
& U\p_x U  + V\p_y U - \p_y^2 U + \p_x P =0,\quad (x,y)\in \R\times \R_+,\\
& \p_x U + \p_y V =0,\quad (x,y)\in \R\times \R_+,\\
& \p_y P =0,\quad  (x,y)\in \R\times \R_+,
\end{align}
\end{subequations}
which are supplemented with the boundary conditions
\begin{subequations} \label{TD:2}
\begin{align}
& [U, V] |_{y = 0} = 0, \quad x\in \R\\
&\lim_{y \rightarrow \infty} (U-y) = A, \quad x\in \R, \qquad \lim_{x \rightarrow \pm\infty} (U - y) = 0, \quad y\in \R_+.
\end{align}
\end{subequations}
Here $[U,V]$ and $P$ are respectively the velocity field and the pressure field of the fluid. The third equation of  \eqref{TD:1} implies that $P$ depends only on the horizontal variable $x$. The key coupling inherent to the Triple-Deck system is the relation that links $A=A(x)$ to the pressure (called the {\it pressure-displacement} relation):
\begin{align} \label{PDR}
P(x) = \frac{1}{\pi} {\rm p.v.} \int_{\mathbb{R}} \frac{\p_x A(x')}{x - x'} dx = |\p_x| A (x). 
\end{align}
Here the operator $|\p_x|$ is formally given by the Fourier mutiplier, i.e., $|\p_x| A = \mathcal{F}^{-1} [|\xi| \hat{A}]$, where $\hat{A}  = \hat{A} (\xi)$ is the Fourier transform of $A$ with respect to the $x$ variable. Then we obtain the system 
\begin{subequations}\label{TD:3}
\begin{align}
& U\p_x U  + V\p_y U - \p_y^2 U =-\p_x|\p_x| A,\quad (x,y)\in \R\times \R_+,\\
& \p_x U + \p_y V =0,\quad (x,y)\in \R\times \R_+,\\
&[U, V]|_{y = 0} = 0, x\in \R\\
&\lim_{y \rightarrow \infty} (U-y) = A, \quad x\in \R, \qquad \lim_{x \rightarrow \pm\infty} (U - y) = 0, \quad y\in \R_+.
\end{align}
\end{subequations}

\subsection{Overview of the Triple-Deck}

Due to the mismatch between the no-slip boundary condition for the Navier-Stokes equations and the no-penetration condition for the Euler equations, there are thin,  $O(\text{viscosity}^{\frac12})$, layers of rapid velocity change in which the flow transitions from the Eulerian behavior in the bulk. These regions are called boundary layers, and are where viscous effects become leading order even when the viscosity is very small. In general, the boundary layer models are derived to effectively describe a viscous flow near a physical boundary (one can think of $\{y = 0\}$ as the physical boundary). We point the reader to Prandtl's original work \cite{Prandtl}, as well as the classical treatise by Schlichting-Gersten \cite{SchGer} for a detailed derivation of the stationary boundary layer theory. 

The most classical boundary layer model is the Prandtl system, which reads as follows: 
\begin{align} \label{stPR:1}
U \p_x U + V \p_y U - \p_y^2 U + \p_x P = 0, \qquad (x, y) \in (0, L) \times \mathbb{R}_+,  \\ \label{stPR:2}
\p_x U + \p_y V = 0, \qquad (x, y) \in (0, L) \times \mathbb{R}_+,\\ \label{stPR:3}
\p_y P = 0, \qquad (x, y) \in (0, L) \times \mathbb{R}_+. 
\end{align} 
Above, $L > 0$ is a parameter that can be taken to be $\infty$. Physically, the boundary $\{x = 0\}$ represents an ``inflow" boundary where one may prescribe the velocity of the flow. The boundary $\{y = 0\}$ is the physical boundary where one prescribes the no-slip condition, and $\{y = \infty\}$ represents the Eulerian regime where one demands the matching condition to the Euler trace: 
\begin{align} \label{stPR:4}
&U|_{x = 0} = U_{\text{INIT}}(y), \\ \label{stPR:5}
&[U, V]|_{y = 0} = 0, \\ \label{stPR:6}
&\lim_{y \rightarrow \infty} U = u_E(x). 
\end{align}
Upon evaluation of \eqref{stPR:1} at $y = \infty$, we obtain the classical Bernoulli's law which relates the pressure to the Euler trace: 
\begin{align} \label{bernoulli:1}
\p_x P(x) = - u_E(x) u_E'(x). 
\end{align}

Formally, the equation \eqref{stPR:1} exhibits a parabolic scaling $U \p_x \sim \p_y^2$. This scaling is consistent with the role of $\{x =0\}$ as an inflow boundary. It is for this reason that the initial datum is provided in \eqref{stPR:4}, and the stationary Prandtl equations are typically regarded as an evolution equation in $x$. The Cauchy problem \eqref{stPR:1} -- \eqref{bernoulli:1}, both local and global, was initiated by Oleinik, \cite{Oleinik}, and has been intensively studied in the last several years (e.g. \cite{Serrin,IyerBlasius,GuoIyer,WZ,YZ,IMRev} as just a few examples).

One of the most important phenomena in boundary layer theory is that of flow separation. In the presence of an adverse pressure gradient $(\p_x p_E(x) > 0)$, the laminar component of the flow may detach from the boundary at a particular point, $X_\ast$, leaving in its wake a turbulent regime. The formation of separation has been proven rigorously for the stationary Prandtl system in the works of Shirota and Matsui \cite{SM1,SM2}, Dalibard-Masmoudi \cite{DM}, and Shen-Wang-Zhang \cite{SWZ}. 

It turns out that near the separation point, $X_\ast$, due to large tangential gradients, the formal asymptotics that allow one to derive the Prandtl system from the Navier-Stokes equations are no longer justifiable. Therefore, it is reasonable to believe that the Prandtl system has to be modified in order to account for these large tangential gradients near the separation point in order to accurately govern the leading order dynamics of the Navier-Stokes flow near the separation point. For this reason, fluid dynamicists proposed a rescaling of the tangential variable, $x$, to $O(\text{viscosity}^{\frac38})$ near the separation point. In so doing, the Triple-Deck equations, \eqref{TD:1} -- \eqref{TD:2}, were derived as a refinement to the Prandtl system. 

We point the reader towards the article of Meyer \cite{Meyer}, which provides a detailed historical overview of the Triple-Deck system. The reader may consult Lagr{\'e}e \cite{Lagree} for an in-depth derivation of the Triple-Deck equations together with a justification for the many rescalings involved in these derivations. Recently, there have been some rigorous mathematical results on the dynamical Triple-Deck system by the first author and Vicol \cite{IV}, Dietert and G{\'e}rard-Varet \cite{DGV}, and the authors and G{\'e}rard-Varet \cite{GVIM}. However, unlike the stationary Prandtl system, the stationary Triple-Deck has thus far not been studied from a rigorous mathematical standpoint to the authors' knowledge. 

Comparing the Triple-Deck equations with the classical Prandtl equations, one finds the following key differences: 
\begin{itemize}
\item[(1)] The pressure relation through the Bernoulli law, \eqref{bernoulli:1}, is ``passive": from the point of view of \eqref{stPR:1} -- \eqref{bernoulli:1}, the quantity $P(x)$ is determined externally by the given Euler flow. In contrast, the Pressure-Displacement relation \eqref{PDR} for the Triple-Deck system is determined ``actively": the quantity $A$ itself is an unknown in the boundary layer. Moreover, by using $\p_x P(x) = \p_x |\p_x|A = \p_x |\p_x| U|_{y = \infty}$, this operation is a relatively singular operation on the unknown $U$.   

\item[(2)] The Prandtl system is set on the tangential interval $x \in (0, L)$, and is importantly supplemented with the Cauchy data, $U_{INIT}(y)$. This Cauchy data is the mechanism by which nontrivial solutions are generated. On the other hand, the stationary Triple-Deck system is set on the tangential interval $x \in \mathbb{R}$ and \textit{it does not have Cauchy data}, and instead, only the condition at $x=-\infty$ is imposed  as $\lim_{x \rightarrow - \infty} U = y$.

Physically, this distinction arises from the rescaling near the separation point, in which the original tangential variable before the separation point get pushed to $-\infty$ under the rescaling. Mathematically, this implies that the stationary Triple-Deck system cannot truly be thought of as an evolution equation, and it is unclear what the mechanism is to generate nontrivial solutions.
\end{itemize}

Point (1) is also present in the dynamical setting and represents the main difficulties of dynamical Triple-Deck compared to the unsteady Prandtl equations (and has thus been encountered in the previously mentioned works \cite{IV,DGV,GVIM}). On the other hand, point (2) above is unique to the stationary setting, and has not appeared in the previous works. In fact, this observation inspires our point-of-view in this article, which we now describe. It is clear that the Couette flow
\begin{align}\label{couette}
[U, V, A] = (y, 0, 0)
\end{align}
is an exact stationary solution to the system \eqref{TD:3}. Given point (2) above, the precise formulation of our question is: do there exist solutions to \eqref{TD:3} which are distinct from the Couette flow? Interestingly, this type of question has been raised even by the fluid-dynamicists who developed the Triple-Deck model. We give the following statement from Meyer's article \cite{Meyer}:
\begin{quote}
One of the main distinctions is that Prandtl's equations are in inhomogeneous system forced by a known external flow, but the [stationary Triple-Deck Equations] are a \textit{homogeneous} system and \textbf{may therefore be suspected of admitting nontrivial solutions only for particular values of an as yet unidentified eigenparameter.}
\end{quote}

The aim of this paper is to show the local rigidity of the Couette flow. That is, we will show the Liouville-type theorem such that, in a suitable functional framework, there exist no other solutions of \eqref{TD:3} around the Couette flow \eqref{couette}. We now describe our main result in more precise terms. 

\subsection{Main Result}

Let us consider perturbations to the Couette flow of the form 
\begin{align}
[U,V,A] = [y+u,v, A].
\end{align}
Then $[u,v,A]$ obeys the perturbed system 
\begin{subequations}\label{ptd:1}
\begin{align}
& y \p_x u + v -\p^2_y u = -\p_x|\p_x| A- u\p_x u - v\p_y u,\quad (x,y)\in \R\times \R_+,\\
& \p_x u + \p_y v =0,\quad (x,y)\in \R\times \R_+,\\
&[u, v]|_{y = 0} = 0, \quad x\in \R,\\
&\lim_{y \rightarrow \infty} u = A, \quad x\in \R,\qquad \lim_{x \rightarrow \pm\infty} u  = 0, \quad y\in \R_+.
\end{align}
\end{subequations}
For the study of the system \eqref{ptd:1} it is useful to introduce the {\it vorticity} field $\omega = \p_y u$, i.e., $u$ and $\omega$ is related as  
\begin{align}
u (x,y) =I_y [\omega] (x,y) := \int_0^y \omega (x,z) dz.
\end{align}
Then, by imposing the enough decay of $\omega$, the condition at $y=\infty$ in the fourth equation of \eqref{ptd:1} is written as the integral condition 
\begin{align}
I_\infty [\omega] =A,\quad x\in \R.
\end{align}
On the other hand, taking the boundary trace in the first equation of \eqref{ptd:1}, we find that $\omega$ is subject to the Neumann type boundary condition
\begin{align}\label{bc.vor}
\p_y \omega|_{y=0} =\p_x |\p_x| A,\quad x\in \R.
\end{align}
Hence, taking $\p_y$ in the first equation of \eqref{ptd:1}, we obtain the system for $\omega$ such as
\begin{subequations}\label{ptd:2} 
\begin{align}
& y \p_x \omega -\p^2_y \omega = - u\p_x \omega - v\p_y \omega ,\quad (x,y)\in \R\times \R_+,\\
& [u,v] = \big [I_y [\omega],-\p_x I_y [u]\big],\quad (x,y)\in \R\times \R_+,\\
& \p_y\omega|_{y=0} = \p_x|\p_x| A, ~~ I_\infty[\omega] = A,\quad  x\in \R.
\end{align}
\end{subequations}
To state the result we first fix the definition of solutions to \eqref{ptd:2} in this paper. Let us denote by $BC(\R)$ the space of all bounded and continuous functions in $\R$, and by $C_0^\infty (\overline{\R^2_+})$ the set of all smooth and compactly supported functions on $\overline{\R^2_+}$.
Let $\langle \cdot,\cdot\rangle_{L^2(\Omega)}$ be the inner product of $L^2(\Omega)$ for a domain $\Omega$. For $Q\in BC(\R)$ with $\p_x Q\in L^3 (\R)\cap L^2 (\R)$ we set $X_Q$ as the Banach space defined by 
\begin{align}\label{df.X_Q}
\begin{split}
X_Q & = \overline{C_0^\infty (\overline{\R^2_+})}^{\|\cdot \|_{X_Q}},\\
\| f\|_{X_Q} & = \| f\|_{L^\infty_x (\R; L^1_y(\R_+))} +  \|\p_y f\|_{L^6_x(\R; L^2_y(\R_+))} + \|\p_y^2f\|_{L^2(\R^2_+)} + \|y (\p_xf - \p_x Q\, \p_yf)\|_{L^2(\R^2_+)}.
\end{split}
\end{align}
%Here $\gamma_{\p\R^2_+}$ is the boundary trace operator.

\begin{remark} (i) One can check that $X_Q$ defined by \eqref{df.X_Q} is a Banach space. In particular, if $\{f_n\}\subset X_Q$ is a Cauchy sequence with respect to the norm $\|\cdot\|_{X_Q}$, then it converges in $X_Q$ and the limit $f$ satisfies $y\p_xf, \p_xQ\, \p_y f\in L^2_{loc}(\overline{\R^2_+})$ and $y (\p_x f-\p_x Q\, \p_y f)\in L^2(\R^2_+)$.

\noindent (ii) When $Q=0$ one can show that $|\p_x|^\frac13 \p_y f\in L^2(\R^2_+)$ for any $f\in X_Q$, which is compatible with $\p_y f \in L^6_x (\R; L^2_y (\R_+))$ by the Sobolev embedding inequality.
\end{remark}

\begin{definition}\label{def.sol} We say that $[\omega,A]$ is an $L^2$ strong solution to \eqref{ptd:2} if the following conditions are satisfied.

\noindent {\rm (i)} $A\in BC (\R)$, $\p_xA\in H^\frac76 (\R)$, and $\omega\in X_{A}$. 

\noindent {\rm (ii)} The couple $[\omega,A]$ satisfies 
\begin{subequations}
\begin{align}
&\begin{split}
\langle y\p_x\omega, \varphi\rangle_{L^2(\R^2_+)} + \langle \p_y\omega,\p_y\varphi\rangle_{L^2(\R^2_+)} & - \langle |\p_x| A, \p_x \varphi (\cdot,0)\rangle_{L^2(\R)} + \langle u\p_x \omega+ v\p_y \omega, \varphi \rangle_{L^2(\R^2_+)} =0\\
& {\rm for ~any} ~~\varphi \in C_0^\infty (\overline{\R^2_+}),
\end{split}\\
& \quad I_\infty[\omega] =A,\quad x\in \R,
\end{align}
\end{subequations}
where $u=I_y[\omega]$ and $v=-\p_xI_y[u]$.
\end{definition}

\begin{remark} (i) The condition $\omega\in L^\infty_x (\R; L^1_y(\R_+))$ in $X_A$ is a natural requirement for the term $I_\infty[\omega]$ to be well-defined.

\noindent (ii) The boundary condition is taken into account in the sense of distributions.

\noindent (iii)  The term $\langle u\p_x \omega, \varphi\rangle_{L^2(\R^2_+)}$ is well-defined for $\omega\in X_A$, for 
\begin{align*}
| \langle u\p_x \omega, \varphi\rangle_{L^2(\R^2_+)} | \leq \|y\p_x \omega\|_{L^2} \| y^{-1}u\varphi\|_{L^2} & \leq C \| y\p_x\omega\|_{L^2} \| \p_y (u \varphi) \|_{L^2}
\end{align*}
by the Hardy inequality, and 
\begin{align*}
\| \p_y (u\varphi)\|_{L^2} \leq \| \omega \varphi \|_{L^2} + \| u\p_y \varphi \|_{L^2}. 
\end{align*}
Since $\|\omega (x,\cdot)\|_{L^\infty_y}^2 \leq 2\| \p_y \omega (x,\cdot)\|_{L^\infty_y} \| \omega (x,\cdot) \|_{L^1_y}$ and $\| \p_y \omega (x,\cdot)\|_{L^\infty_y}^2 \leq 2 \| \p_y^2 \omega (x,\cdot) \|_{L_y^2} \| \p_y \omega (x,\cdot) \|_{L^2_y}$, we have $\p_y\omega \in L^3_x (\R; L^\infty_y(\R_+))$ and $\omega \in L^6_x (\R; L^\infty_y (\R_+))$ with the estimate
\begin{align*}
\| \omega \|_{L^6_xL^\infty_y} \leq C \| \p_y^2 \omega\|_{L^2}^\frac14 \| \p_y\omega \|_{L^6_xL^2_y}^\frac14 \| \omega \|_{L^\infty_xL^1_y}^\frac12.
\end{align*}
We also have $\|u\|_{L^\infty} \leq \|\omega\|_{L^\infty_xL^1_y}$, and hence, 
\begin{align*}
\| \p_y (u\varphi) \|_{L^2} \leq C \| \p_y^2 \omega\|_{L^2}^\frac14 \| \p_y\omega \|_{L^6_xL^2_y}^\frac14 \| \omega \|_{L^\infty_xL^1_y}^\frac12 \| \varphi \|_{L^3_xL^2_y} + \| \omega \|_{L^\infty_xL^1_y} \| \p_y\varphi \|_{L^2}.
\end{align*}

\noindent (iv) The term $\langle v\p_y \omega, \varphi\rangle_{L^2(\R^2_+)}$ is well-defined for $\omega\in X_A$. Indeed, since $I_y[u] = y A - I_y [I_\cdot^\infty [\omega]]$ with 
\begin{align*}
I_y^\infty[\omega] (x,y) = \int_y^\infty \omega (x,z)\,  dz,
\end{align*}
we can verify the identity  
\begin{align}\label{eq.v}
v = -\p_x I_y[u] = - y \p_x A +  I_y[I_{\cdot}^\infty \big [\p_x \omega -\p_x A\, \p_y \omega] \big] - \p_x A\, I_y [\omega].
\end{align}
Then we have 
\begin{align*}
|v(x,y)|\leq (y + \|\omega\|_{L^\infty_xL^1_y})  |\p_x A(x)| + C y^\frac12 \| z \big ( \p_x\omega - \p_x A\, \p_z \omega\big )(x,\cdot)\|_{L^2_z (\R_+)},
\end{align*}
where the estimate $|I_y^\infty[\p_x\omega -\p_x A\,\p_z\omega] (x)| \leq C y^{-\frac12} \|z(\p_x\omega -\p_x A\, \p_z \omega )(x,\cdot)\|_{L^2_z(\R_+)}$ is used. Thus, 
\begin{align*}
& |\langle v\p_y \omega, \varphi\rangle_{L^2(\R^2_+)}| \\
&\leq \|\p_y \omega\|_{L^6_xL^2_y} \| v \varphi\|_{L^\frac65_xL^2_y} \\
& \leq C  \| \p_y \omega\|_{L^6_xL^2_y} \big ( (1+\|\omega\|_{L^\infty_xL^1_y}) \|\p_xA\|_{L^2_x} + \|y(\p_x\omega - \p_xA\, \p_y \omega)\|_{L^2_{x,y}} \big ) \| (y+1) \p_y \varphi \|_{L^3_x L^\infty_y}.
\end{align*}
The above calculation is easily justified for $\omega\in C_0^\infty (\overline{\R^2_+})$, and then it is extended to each $\omega\in X_{A}$ by the density argument.

\noindent (v) The condition $\p_x A\in H^\frac76(\R)$ is related to the optimal regularity of the boundary-value problem
\begin{align}\label{blinear:1}
y\p_x w -\p_y^2 w=0,~(x,y)\in \R^2_+,\qquad \p_y w|_{y=0} = g, ~x\in \R.
\end{align}
By applying the Fourier transform in $x$, the solution to \eqref{blinear:1} decaying at infinity is expressed in terms of the Airy function. From this solution formula and the Plancherel theorem, one finds that the requirement $\p_y^2 w\in L^2(\R^2_+)$ is equivalent with $g\in \dot{H}^\frac16 (\R)$. 
\end{remark}

The result of this paper is stated as follows.

\begin{theorem}\label{main.thm:1} There exists $\delta>0$ such that the following statement holds. Let $[\omega,A]$ be an $L^2$ strong solution to \eqref{ptd:2}. If in addition $\omega \in L^\infty (\R^2_+)$, $y( \p_x \omega - \p_x A\, \p_y \omega )\in L^{\frac{12}{5},\infty}_x(\R; L^{\frac43,\infty}_y(\R_+))$, and   
\begin{align}\label{est.thm:1:1}
\begin{split}
&  \| \omega\|_{L^\infty_x (\R; L^1_y (\R_+))} \| \p_x A\|_{L^3(\R)} +  \|\omega\|_{L^\infty(\R^2_+)} \\
& \quad + \|\p_y \omega\|_{L^6_x(\R; L^2_y(\R_+))} + \| y( \p_x \omega - \p_x A\, \p_y \omega )\|_{L^{\frac{12}{5},\infty}_x(\R; L^{\frac43,\infty}_y(\R_+))}  \leq \delta,
\end{split}
\end{align} 
then $[\omega, A]=[0,0]$.
\end{theorem}
 
In \eqref{est.thm:1:1} the space $L^{\frac43,\infty}_y (\R_+)$ is the weak $L^\frac43$ space on $\R_+$, and $L^{\frac{12}{5},\infty}_x(\R; Y)$ is the weak $L^\frac{12}{5}$ space on $\R$ with valued in the Banach space $Y$.
The smallness condition \eqref{est.thm:1:1} is related with the scaling of \eqref{ptd:2}. To be precise, let us introduce the scaling:
\begin{align}\label{scale}
\omega_\lambda (x,y) = \omega (\lambda^3 x,\lambda y), \quad u_\lambda (x,y) = \lambda^{-1} u (\lambda^3 x,\lambda y), \quad v_\lambda (x,y) = \lambda^2 v (\lambda^3 x,\lambda y), \quad \lambda>0.
\end{align}
This scaling is invariant in the sense that, if $[\omega,u,v]$ solves the first equation of \eqref{ptd:2}, then so does $[\omega_\lambda,u_\lambda,v_\lambda]$. The following norms are invariant under the scaling \eqref{scale}:
\begin{align*}
& \|\omega_\lambda\|_{L^\infty(\R^2_+)}=\|\omega\|_{L^\infty(\R^2_+)},\quad \| y\p_x\omega_\lambda \|_{L^2(\R^2_+)} = \| y\p_x\omega \|_{L^2(\R^2_+)},\quad \|\p_y^2 \omega_\lambda \|_{L^2(\R^2_+)}=\|\p_y^2\omega\|_{L^2(\R^2_+)},\\
& \|\p_y \omega_\lambda\|_{L^6_x(\R; L^2_y(\R_+))}  = \|\p_y \omega\|_{L^6_x(\R; L^2_y(\R_+))} ,\quad \|y\p_x\omega_\lambda\|_{L^{\frac{12}{5},\infty}_x(\R; L^{\frac43,\infty}_y(\R_+))} =\| y \p_x\omega \|_{L^{\frac{12}{5},\infty}_x(\R; L^{\frac43,\infty}_y(\R_+))}  .
\end{align*}
For $A$ the system \eqref{ptd:2} does not possess any invariant scaling. Indeed, when $\omega_\lambda (x,y) =\omega(\lambda^3 x,\lambda y)$ is introduced, the integral condition $I_\infty[\omega]=A$ leads to the scaling $\lambda\mapsto \lambda^{-1} A(\lambda^3 x)$, while the boundary condition yields the scaling $\lambda\mapsto \lambda^{-5} A (\lambda^3 x)$. This mismatch is related to the fact that, for $A$, the elliptic equation of the order $|\p_x|^\frac43$ is inherent in the system \eqref{ptd:2}. It should be noted that, if the scaling $A_\lambda (x) = \lambda^{-1} A(\lambda^3x)$ is introduced, we have the following invariance:
\begin{align*}
 \|\omega_\lambda\|_{L_x^\infty(\R;L^1_y(\R_+))} \|\p_x A_\lambda \|_{L^3(\R)} & =\|\omega\|_{L^\infty_x (\R; L^1_y (\R_+))} \|\p_x A \|_{L^3(\R)}, \\
 \| y \p_x A_\lambda \, \p_y \omega_\lambda \|_{L^{\frac{12}{5},\infty}_x(\R; L^{\frac43,\infty}_y(\R_+))} & = \| y \p_x A\, \p_y \omega \|_{L^{\frac{12}{5},\infty}_x(\R; L^{\frac43,\infty}_y(\R_+))}.
\end{align*} 
In this sense, the condition \eqref{est.thm:1:1} is described just in terms of the scale invariant quantities. As far as the authors know, Theorem \ref{main.thm:1} is the first result which reveals the scale critical quantities leading to the Liouville-type theorem for the stationary Triple-Deck system.
Our motivation for working in scale invariant spaces is in order to maximize the space we prove our result in and also to reveal a deep mathematical structure of the stationary Triple-Deck system. Indeed, assuming smallness in a stronger space would result in a easier result and essentially follow from the simpler linear estimates but with less understanding of the structure of the equations. On the other hand, to achieve the rigidity in a wider space than the scale-critical regime would be substantially harder, and in fact would immediately imply the \textit{global} rigidity of the Couette flow. So far, this is an important open problem.

In the proof of Theorem \ref{main.thm:1} the following observations are essential, and these are the key contributions of this paper for the mathematical analysis of the stationary Triple-Deck system.

\vspace{0.3cm}

(1) Transport effect from $y\p_x \omega +v\p_y \omega$: Apparently the leading term of the first equation of \eqref{ptd:2} is $y\p_x\omega -\p_y^2 \omega$, for we are imposing a smallness for the unknowns. However, this is not correct in the mathematical treatment, due to the presence of the term $-v\p_y\omega$, where $v$ behaves like $-y \p_xA$ for $y\gg 1$. Such a linear growth in $y$ prevent us treating the term $-v\p_y \omega$ as a perturbation for $y\gg 1$ and this term has to be taken as the principal part of the system in the rigorous analysis. To overcome this difficulty we introduce the nonlinear transformation $w(x,y) = \omega (x,\Phi (x,y))$ with $\Phi (x,y) = y-\chi_R (y) A(x)$, where $\chi_R (y) = \chi (\frac{y}{R\|\p_y \chi\|_{W^{1,\infty}} (1+\|A\|_{\infty})})$ for $R\geq 2$ and $\chi$ is a smooth cut-off such that $\chi(y)=0$ for $0\leq y\leq 1$ and $\chi(y)=1$ for $y\geq 2$; see Section \ref{sec.reduction}. This transformation is introduced by taking into account the balance $y\p_x \omega$ in the left-hand side of \eqref{ptd:2} and $-v\p_y \omega \sim y\p_x A \, \p_y \omega$ in the right-hand side of \eqref{ptd:2}, which implies that the stationary Triple-Deck system contains the nonlinear transport operator $\p_x  - c\, \p_y$ with $c=c(x) =\p_x A(x)$ in the regime $y\gg 1$. The above transformation is compatible with this transport effect. It should be also emphasized here that the choice of the size of the cut-off depending on the unknown $\|A\|_{\infty}$ is important to achieve the local rigidity in the scale critical framework as in Theorem \ref{main.thm:1}.

\vspace{0.3cm}

(2) $|\p_x|^\frac43$ ellipticity for $A$ and avoidance of possible derivative loss: It is clearly crucial to understand the role of the unkown $A$ in the analysis of the Triple-Deck system, which is the key difference from the classical Prandtl equations. In this paper we derive from the vorticity formulation \eqref{ptd:2} the elliptic equation of $A$ with the elliptic order $\frac43$; see Subsection \ref{subsec.A}. In this step the use of  \eqref{ptd:2} rather than \eqref{ptd:1} is important. Indeed, it seems to be difficult to find this structure for $A$ from the velocity formulation \eqref{ptd:1}. The elliptic smoothing for $A$ with the order $\frac43$ will be optimal in view of the formal scaling argument and is essential in the proof of Theorem \ref{main.thm:1}. It is also crucial to avoid possible derivative loss in $x$ for the system. Indeed, the term $v\p_y \omega$ can be a sorce of the derivative loss in $x$ due to the nonlocal relation $v=-\p_x I[u]$. Although the kinetic operator $y\p_x-\p_y^2$ of the system implies the smoothing in $x$ of the order $1$ for $y\geq 1$ and of the order $\frac23$ for $y>0$, this is not sufficient to handle the term $\p_x u$ near the boundary. 
The key point here is the smoothing of $A$ together with the formula \eqref{eq.v} for $v$, which enables us to escape from the derivative loss in $x$, resulting in Theorem \ref{main.thm:1} within a natural scale framework.

\

The rest of this paper is organized as follows. In Section \ref{sec.pre} we collect some basic functional inequalities that will be used throughout our analysis. In particular, the key estimates related with the kinetic operator $y\p_x -\p_y^2$ in the half space are stated. In Section \ref{sec.reduction} we introduce the transformation of the unknown $\omega$ and analyze the reduction system as well as the elliptic equation of $A$. The proof of Theorem \ref{main.thm:1} is then completed in Section \ref{sec.proof.main}. 

\section{Preliminaries}\label{sec.pre}

\subsection{Basic inequalities in the Lorentz space}

Let $\R^d$, $d\in \N$, be the $d$ dimensional Euclidean space. The following two propositions by O'Neil \cite{O'N} will be freely used in this paper. 

\begin{proposition}[Generalized H${\rm \ddot{o}}$lder inequality for product]\label{prop.pre.product} Let $\Omega \subset \R^d$ be a Lebesgue measurable set. Let $1< p,p_1,p_2<\infty$ and $1\leq q, q_1,q_2\leq \infty$ satisfy $\frac1p=\frac{1}{p_1}+\frac{1}{p_2}$ and $\frac1q=\frac{1}{q_1}+\frac{1}{q_2}$. Then, 
\begin{align}
\| fg\|_{L^{p,q}(\Omega)}\leq C \| f\|_{L^{p_1,q_1}(\Omega)} \| g\|_{L^{p_2,q_2}(\Omega)}.
\end{align}
\end{proposition}

\begin{proposition}[Generalized Young's inqeuality for convolution]\label{prop.pre.convo} Let $1< p,p_1,p_2<\infty$ and $1\leq q, q_1,q_2\leq \infty$ satisfy $\frac1p=\frac{1}{p_1}+\frac{1}{p_2}-1$ and $\frac1q=\frac{1}{q_1}+\frac{1}{q_2}$. Then, 
\begin{align}
\| f*g\|_{L^{p,q}(\R^d)} \leq C\|f\|_{L^{p_1,q_1}(\R^d)} \| g\|_{L^{p_2,q_2} (\R^d)}.
\end{align}
\end{proposition}

\subsection{Properties of the class $X_Q$}

\begin{lemma}\label{lem.pre:0}  Let $Q\in BC(\R)$ with $\p_xQ\in L^3(\R)\cap L^2(\R)$. If $f\in X_Q$, then 
\begin{align}\label{est.lem.pre:0}
\lim_{|x|\rightarrow \infty} \| f(x,\cdot)\|_{L^1_y(\R_+)}= 0. 
\end{align}
Conversely, if $f\in BC_x (\R; L^1_y (\R_+))$,  $\p_y f\in L^6_x (\R; L^2_y(\R_+))$,  $\p_y^2 f, \, y(\p_x f-\p_xQ \, \p_yf) \in L^2(\R^2_+)$, and \eqref{est.lem.pre:0} holds, then $f\in X_Q$.
\end{lemma}

The proof of Lemma \ref{lem.pre:0} is postponed to the appendix. In fact, the proof for the {\it converse} part in Lemma \ref{lem.pre:0} becomes a bit technical, due to the term $y \p_xQ \, \p_yf$.

\begin{proposition}\label{lem.pre:1} There exists $C>0$ such that
\begin{align}\label{est.lem.pre:1:1}
\| |\p_x|^\frac16 \gamma_{\p\R^2_+} \p_y f\|_{L^2(\R)} \leq C \big (\| \p_y^2 f\|_{L^2(\R^2_+)} + \| y\p_x f\|_{L^2 (\R^2_+)} \big )
\end{align}
and 
\begin{align}\label{est.lem.pre:1:2}
\|f\|_Y \leq C \big ( \| y\p_x f -\p_y^2 f\|_{L^2(\R^2_+)} + \| |\p_x|^\frac16 \gamma_{\p\R^2_+} \p_y f\|_{L^2(\R)} \big )
\end{align}
for any $ f \in X_Q$ with $Q\equiv 0$. 
Here  $\gamma_{\p\R^2_+}$ is the boundary trace operator and 
\begin{align}\label{est.lem.pre:1:3}
\begin{split}
\|f\|_Y & = \|\p_y^2 f\|_{L^2 (\R^2_+)} + \| y\p_xf \|_{L^2(\R^2_+)} + \| |\p_x|^\frac23 f\|_{L^2(\R^2_+)} +\| |\p_x|^\frac13 \p_y f\|_{L^2(\R^2_+)} \\
& \quad + \| |\p_x|^\frac{1}{2} y^\frac12 \p_y f\|_{L^2(\R^2_+)}  + \| |\p_x|^\frac{5}{12} y^\frac14 \p_y f\|_{L^2(\R^2_+)}.
\end{split}
\end{align}
\end{proposition}

\begin{proof} It suffices to show \eqref{est.lem.pre:1:1} and \eqref{est.lem.pre:1:2} for $f\in C_0^\infty (\overline{\R^2_+})$ by the density argument. To prove \eqref{est.lem.pre:1:2}, we use the argument by Bouchut \cite{Bou2}, where the whole space case is discussed and the key observation is the commutator relation $[\p_y,y\p_x]=\p_x$. Set $h= y\p_x f -\p_y^2 f$, and we take the Fourier transform in $x$, which gives 
\begin{align}\label{proof.pre:1:1}
i\xi y \hat{f} (\xi,y) - \p_y^2\hat{f} (\xi,y) = \hat{h} (\xi,y).
\end{align}
In view of \eqref{est.lem.pre:1:3} it is useful to introduce the quantity $B\geq 0$ defined as 
\begin{align*}
B^2=B^2(\xi) = |\xi|^2\| y\hat{f}(\xi,\cdot) \|_{L^2(\R_+)}^2 + |\xi|^\frac43 \| \hat{f}(\xi,\cdot)\|_{L^2(\R_+)}^2.
\end{align*}
Taking the $L^2_y$ inner product with $\xi y\hat{f}$ in \eqref{proof.pre:1:1} implies 
\begin{align}\label{proof.pre:1:2}
i |\xi|^2 \| y\hat{f}\|_{L^2_y}^2 + \xi \| y^\frac12 \p_y \hat{f}\|_{L^2_y}^2 + \xi \langle \p_y \hat{f}, \hat{f}\rangle_{L^2_y}  = \langle \hat{h}, y\xi \hat{f}\rangle_{L^2_y}.
\end{align}
Hence, the imaginary part of \eqref{proof.pre:1:2} yields 
\begin{align}\label{proof.pre:1:3}
|\xi|^2 \| y\hat{f}\|_{L^2_y}^2 \leq |\xi| \|\p_y \hat{f}\|_{L^2_y} \| \hat{f}\|_{L^2_y} + \| \hat{h}\|_{L^2_y} |\xi| \| y\hat{f}\|_{L^2_y} \leq ( |\xi|^\frac13 \| \p_y\hat{f}\|_{L^2_y}  + \| \hat{h}\|_{L^2_y} )B.
\end{align}
Next we have 
\begin{align}\label{proof.pre:1:4}
|\xi|^\frac43 \|\hat{f}\|_{L^2_y}^2 = |\xi|^\frac43 \langle \hat{f}, [\p_y,y]\hat{f}\rangle_{L^2_y}  = - |\xi|^\frac43 \big ( \langle \p_y \hat{f}, y \hat{f}\rangle_{L^2_y} + \langle y \hat{f}, \p_y\hat{f}\rangle_{L^2_y} \big ) & \leq 2 |\xi|^\frac13 \| \p_y \hat{f}\|_{L^2_y} |\xi| \| y \hat{f}\|_{L^2_y}\nonumber \\
& \leq 2 |\xi|^\frac13 \|\p_y \hat{f} \|_{L^2_y} B.
\end{align}
Finally we take the $L^2_y$ inner product with $|\xi|^\frac23 \hat{f}$ in \eqref{proof.pre:1:1}, which yields 
\begin{align}\label{proof.pre:1:5}
i\xi |\xi|^\frac23 \|y^\frac12 \hat{f}\|_{L^2_y}^2 + |\xi|^\frac23 \| \p_y \hat{f}\|_{L^2_y}^2 = - |\xi|^\frac23 \p_y \hat{f}|_{y=0} \overline{\hat{f}}|_{y=0} + |\xi|^\frac23 \langle\hat{h},\hat{f}\rangle_{L^2_y}.
\end{align}
Then the real part of \eqref{proof.pre:1:5} gives 
\begin{align*}
|\xi|^\frac23 \| \p_y \hat{f}\|_{L^2_y}^2 & \leq  |\xi|^\frac23 |\p_y \hat{f}|_{y=0}| \, |{\hat{f}}|_{y=0}| +\|\hat{h}\|_{L^2_y} |\xi|^\frac23 \|\hat{f}\|_{L^2_y} \nonumber \\
& \leq  |\xi|^{\frac16} |\p_y \hat{f}|_{y=0}| (2 |\xi|^\frac13\| \p_y \hat{f}\|_{L^2_y} |\xi|^\frac23 \|\hat{f}\|_{L^2_y})^\frac12  +\|\hat{h}\|_{L^2_y} B\nonumber \\
& \leq C (|\xi|^\frac16 |\p_y \hat{f}|_{y=0}|)^\frac43 (|\xi|^\frac23 \| \hat{f}\|_{L^2_y})^\frac23 + \frac12 |\xi|^\frac23 \|\p_y\hat{f}\|_{L^2_y}^2   +\|\hat{h}\|_{L^2_y} B,
\end{align*}
that is,
\begin{align}\label{proof.pre:1:6}
|\xi|^\frac23 \| \p_y \hat{f}\|_{L^2_y}^2 & \leq C (|\xi|^\frac16 |\p_y \hat{f}|_{y=0}|)^\frac43 (|\xi|^\frac23 \| \hat{f}\|_{L^2_y})^\frac23  + C\|\hat{h}\|_{L^2_y} B \leq C (|\xi|^\frac16 |\p_y \hat{f}|_{y=0}|)^\frac43 B^\frac23 +C\|\hat{h}\|_{L^2_y} B.
\end{align}
Hence, \eqref{proof.pre:1:3}, \eqref{proof.pre:1:4}, and \eqref{proof.pre:1:6} yield
\begin{align*}
B^2 & \leq (3 |\xi|^\frac13 |\p_y\hat{f}\|_{L^2_y} + \|\hat{h}\|_{L^2_y} ) B \nonumber \leq C(|\xi|^\frac16 |\p_y \hat{f}|_{y=0}|)^\frac23  B^\frac43 + C \| \hat{h}\|_{L^2_y}^\frac12 B^\frac32 + C \|\hat{h}\|_{L^2_y} B, 
\end{align*}
which implies 
\begin{align}\label{proof.pre:1:7}
B^2 \leq C \big (|\xi|^\frac16 |\p_y \hat{f}|_{y=0}|)^2 + \|\hat{h}\|_{L^2_y}^2\big ).
\end{align}
Then, \eqref{proof.pre:1:6}, \eqref{proof.pre:1:1}, and the real part of \eqref{proof.pre:1:2} give
\begin{align}\label{proof.pre:1:8}
|\xi|^\frac23\|\p_y\hat{f}\|_{L^2_y}^2 + \|\p_y^2\hat{f}\|_{L^2}^2  + |\xi| \| y^\frac12 \hat{f}\|_{L^2_y}^2 \leq C \big ( (|\xi|^\frac16 |\p_y \hat{f}|_{y=0}|)^2 + \|\hat{h}\|_{L^2_y}^2\big ).
\end{align}
Hence, \eqref{est.lem.pre:1:2} follows from the Plancherel theorem, together with the interpolation inequality
\begin{align}\label{proof.pre:1:9}
|\xi|^\frac56 \| y^\frac14 \p_y\hat{f}\|_{L^2_y}^2 \leq |\xi|^\frac12 \|y^\frac12 \p_y \hat{f}\|_{L^2_y}\,  |\xi|^\frac13  \| \p_y \hat{f}\|_{L^2_y}.
\end{align}
Next we show \eqref{est.lem.pre:1:1}. We observe that 
\begin{align*}
|\xi|^\frac13 |\p_y \hat{f}|_{y=0}|^2 & \leq 2 |\xi|^\frac13 \|\p_y^2\hat{f}\|_{L^2_y}\| \p_y\hat{f}\|_{L^2_y}\\
& \leq C  \|\p_y^2\hat{f}\|_{L^2_y} \big ( |\xi|^\frac16 |\p_y \hat{f}|_{y=0}| + \|\hat{h}\|_{L^2_y}\big ),
\end{align*}
where \eqref{proof.pre:1:8} is used in the last line. This gives 
\begin{align}\label{proof.pre:1:10}
|\xi|^\frac13 |\p_y \hat{f}|_{y=0}|^2 \leq C \|\p_y^2\hat{f}\|_{L^2_y} \big ( \|\p_y^2\hat{f}\|_{L^2_y}  + \|\hat{h}\|_{L^2_y}\big )\leq C \big ( \|\p_y^2\hat{f}\|_{L^2_y}^2  + |\xi|^2 \| y\hat{f}\|_{L^2_y}^2\big ).
\end{align}
Thus, \eqref{est.lem.pre:1:1} follows from the Plancherel theorem. The proof is complete.
\end{proof}

\begin{corollary}\label{cor:lem.pre:1} There exists $C>0$ such that 
\begin{align}\label{est.cor.lem.pre:1}
\| \p_y f\|_{L^{6,2}_x (\R; L^2_y (\R_+))}  + \| y^\frac14 \p_y f\|_{L^{12,2}_x(\R; L^2(\R_+))} \leq C \big ( \| y\p_x f -\p_y^2 f\|_{L^2(\R^2_+)} + \| |\p_x|^\frac16 \gamma_{\p\R^2_+} \p_y f\|_{L^2(\R)} \big )
\end{align}
for any $f\in X_Q$ with $Q\equiv 0$. Here $L^{p,q}(\R; Y)$ with $1\leq p,q\leq \infty$ is the Lorentz space with valued in the Banach space $Y$.
\end{corollary}

\begin{proof} Since the operator $|\p_x|^{-\frac13}$ and $|\p_x|^{-\frac{5}{12}}$ are respectively expressed in terms of the convolution kernel $C|x|^{-1+\alpha}$ with $\alpha=\frac13$ and $\alpha=\frac{5}{12}$, we have the embedding $\dot{H}^\frac13(\R)\hookrightarrow L^{6,2}(\R)$ and $\dot{H}^\frac{5}{12}(\R)\hookrightarrow L^{12,2}(\R)$, as stated in O'Neil \cite{O'N} (indeed, these are the consequences of Proposition \ref{prop.pre.convo}). Hence, \eqref{est.cor.lem.pre:1} follows from \eqref{est.lem.pre:1:2} of Proposition \ref{lem.pre:1}. The proof is complete.
\end{proof}

\begin{proposition}\label{prop.pre:1} Let $g\in H^\frac16(\R)\cap \dot{H}^{-1}(\R)$. Then there exists a unique solution $w_b\in X_Q$ with $Q\equiv 0$ to the equation
\begin{align}\label{eq.b}
\langle y\p_xw_b, \varphi\rangle_{L^2(\R^2_+)} + \langle \p_yw_b,\p_y\varphi\rangle_{L^2(\R^2_+)} & + \langle g, \varphi (\cdot,0)\rangle_{L^2(\R)}  =0\quad  {\rm for ~any} ~~\varphi \in C_0^\infty (\overline{\R^2_+}).
\end{align}
In particular,  $\| w_b \|_Y\leq C \| |\p_x|^\frac16 g \|_{L^2(\R)}$ holds by Lemma \ref{lem.pre:1}. Moreover, $w_b$ satisfies 
\begin{align}
\| \p_y w_b \|_{L^\infty_y (\R_+; L^2_x(\R))} \leq C \| g\|_{L^2 (\R)}.
\end{align}
\end{proposition}

\begin{proof} The uniqueness in $X_Q$ with $Q\equiv 0$ follows from \eqref{est.lem.pre:1:2}. To show the existence, we denote by ${\rm Ai} (z)$ the Airy function, which has the series expansion:
\begin{align*}
{\rm Ai} (z) = \sum_{n=0}^\infty \big [\frac{z^{3n}}{3^{2n+\frac23} \Gamma (n+\frac23)n!} - \frac{z^{3n+1}}{3^{2n+\frac43}\Gamma (n+\frac43)n!}\big ], \quad z\in \mathbb{C}.
\end{align*}
It is well known that ${\rm Ai}(z)$ satisfies
\begin{align*}
\frac{d^2}{dz^2} {\rm Ai} (z) - z {\rm Ai} (z) =0,\qquad {\rm Ai} (0)=\frac{1}{3^\frac23 \Gamma (\frac23)},\quad {\rm Ai}'(0) = - \frac{1}{3^\frac32 \Gamma (\frac43)}.
\end{align*}
Moreover, ${\rm Ai}(z)$ has the representation for $|\arg z|<\pi$ such as 
\begin{align*}
{\rm Ai} (z) = \frac{1}{2\pi} e^{-\frac23 z^\frac32} \int_0^\infty e^{-\sqrt{z} r} \cos (\frac13 r^\frac32) \, \frac{dr}{\sqrt{r}},
\end{align*}
from which the following asymptotic expansion is verifed:
\begin{align}\label{proof.prop.pre:1:1}
{\rm Ai} (z) \sim \frac{1}{2\pi z^\frac14} e^{-\frac23 z^\frac32} \sum_{n=0}^\infty \frac{\Gamma (3n+\frac12) (-1)^n}{3^{2n} (2n)! z^{\frac32 n}},\quad z\rightarrow \infty ~{\rm in} ~|\arg z|<\pi.
\end{align}
In particular, ${\rm Ai} (z)$ decays exponentially when $|\arg z|<\pi/3$. 
Now let us set 
\begin{align}\label{proof.prop.pre:1:2}
\hat{w}_b (\xi,y) = 
\begin{cases}
& \displaystyle \frac{{\rm Ai} (e^{\frac{\pi}{6}i} \xi^\frac13 y)}{{\rm Ai}'(0) e^{\frac{\pi}{6}i} \xi^\frac13} \hat{g}(\xi),\quad \xi>0,\\
& \mbox{}\\
& \displaystyle \frac{{\rm Ai} (e^{-\frac{\pi}{6}i} (-\xi)^\frac13 y)}{{\rm Ai}'(0) e^{-\frac{\pi}{6}i} (-\xi)^\frac13} \hat{g}(\xi),\quad \xi<0,
\end{cases}
\end{align}
which decays exponentially together with its derivatives as $y\rightarrow \infty$ by the definition and satisfies for each $\xi \in \R\setminus\{0\}$,
\begin{align*}
\p_y^2 \hat{w}_b =i \xi y \hat{\omega}_b,\quad y>0,\qquad \p_y \hat{w}_b|_{y=0} (\xi) = \hat{g}(\xi).
\end{align*}
Let $w_b$ be the inverse Fourier transform of $\hat{w}_b$ with respect to $\xi$.
We observe that 
\begin{align}\label{proof.prop.pre:1:3}
|\hat{w}_b (\xi,y)|\leq C |\xi|^{-\frac13} \min\{ |\xi|^{-\frac{1}{12}} y^{-\frac14}, |\xi|^{-\frac12} y^{-\frac32}\} \, |\hat{g} (\xi)|,
\end{align}
where we have used the fact that $|\xi^\frac13 y|^l {\rm Ai}(e^{\frac{\pi}{6}i} \xi^\frac13 y)$ for $\xi>0$ and $|(\xi)^\frac13 y|^l {\rm Ai}(e^{-\frac{\pi}{6}i} (-\xi)^\frac13 y)$ for $\xi<0$ are uniformly bounded for each $l\in \N\cup\{0\}$. Then we have from \eqref{proof.prop.pre:1:2} that 
\begin{align}\label{proof.prop.pre:1:4}
\int_\R |\hat{w}_b (\xi,y)| \, d\xi & \leq C \Big (\int_0^\infty \frac{\min\{ |\xi|^{-\frac{1}{6}} y^{-\frac12}, |\xi|^{-1} y^{-3}\} }{|\xi|^\frac23 (1+|\xi|^\frac13)} \frac{|\xi|^2}{1+|\xi|^2}\, d\xi \Big )^\frac12 \| \frac{(1+|\xi|^\frac16)(1+|\xi|)}{|\xi|} \hat{g} (\xi) \|_{L_\xi^2(\R)} \nonumber \\
& \leq \frac{C}{y^\frac14+y^\frac32} \| g\|_{H^\frac16(\R)\cap \dot{H}^{-1}(\R)}. 
\end{align}
This bound implies from the Riemann-Lebesgue lemma that, for each $y>0$, the function $w_b(x,y)$ is continuous in $x$ and $\displaystyle \lim_{|x|\rightarrow \infty} w_b (x,y)=0$. 
Moreover, the bound like \eqref{proof.prop.pre:1:4} and the Lebesgue convergence theorem yield
\begin{align*}
\lim_{x\rightarrow x'} \|w_b (x,\cdot)-w_b (x',\cdot)\|_{L^1_y(\R_+)} =0, \quad \lim_{|x|\rightarrow \infty} \| w_b (x,\cdot) \|_{L^1_y (\R_+)} =0.
\end{align*}
Hence, $w_b\in BC_x (\R; L^1_y (\R_+))$ and \eqref{est.lem.pre:0} holds.
It is easy to see from \eqref{proof.prop.pre:1:2} that $\|\p_y \hat{w}_b (\cdot,y)\|_{L^2_\xi(\R)}\leq C \| g\|_{L_x^2(\R)}$, which gives $\|\p_y w_b \|_{L^\infty_y (\R_+; L^2_x(\R))}\leq C\|g\|_{L^2_x(\R)}$. 
Next we see from \eqref{proof.prop.pre:1:2} that
\begin{align*}
|y\xi \hat{w}_b (\xi,y)| + |\xi|^\frac13 |\p_y\hat{w}_b (\xi,y)|\leq C |\xi|^\frac13 \min\{ |\xi|^{-\frac{1}{12}} y^{-\frac14}, |\xi|^{-\frac12} y^{-\frac32}\} \, |\hat{g} (\xi)|,
\end{align*}
which gives 
\begin{align}\label{proof.prop.pre:1:5}
\int_0^\infty y^2 \xi^2 |\hat{w}_b (\xi,y)|^2 + |\xi|^\frac23 |\p_y \hat{w}_b (\xi,y)|^2 \, dy\leq C |\xi|^{\frac13} |\hat{g} (\xi)|^2,
\end{align}
and therefore, $\| y \p_x w_b \|_{L^2(\R^2_+)}\leq C\| |\p_x|^\frac16 g \|_{L^2_x(\R)}$ and $\| \p_y w_b \|_{L^6_x (\R; L^2_y(\R_+))} \leq C \| |\p_x|^\frac13 \p_y w_b \|_{L^2(\R^2_+)}\leq C\| |\p_x|^\frac16 g \|_{L^2_x(\R)}$. Since $\p_y^2w_b =y\p_xw_b$, we also have the bound for $\p_y^2w_b$ in $L^2(\R^2_+)$. Hence, $w_b\in X_Q$ with $Q\equiv 0$ by Lemma \ref{lem.pre:0}, and $w_b$ is the solution to \eqref{eq.b} by its construction. The proof is complete.
\end{proof}

\begin{proposition}\label{prop.pre:2} Assume that $Q\in BC(\R)$ and $\p_xQ\in L^3(\R)\cap L^2(\R)$. Let $f\in X_Q$. Assume that $\gamma_{\p\R^2_+} \p_y f\in L^2(\R)$. Then the function $h$ defined by 
\begin{align}
h(x,y) = f (x,\Phi(x,y)),\quad \Phi (x,y) = y - \chi_R (y) Q(x)
\end{align} 
belongs to $X_C$ with $C\equiv 0$. Here,
\begin{align*}
\chi_R (y) = \chi (\frac{y}{R \|\p_y\chi\|_{W^{1,\infty}}(1+\|Q\|_\infty)})
\end{align*}
with $R\geq 2$, and $\chi$ is a smooth cut-off such that $\chi (y)=0$ for $0\leq y\leq 1$ and $\chi(y)=1$ for $y\geq 2$. 
\end{proposition}

\begin{proof} We note that $1/2\leq \p_y \Phi (x,y) \leq 3/2$, and for each $x$ the map $y\mapsto \Phi (x,y)$ is a diffeomorphism on $\R_+$, and the inverse $\eta\mapsto \Phi^{-1} (x,\eta)$ is continuous in $x$. It is easy to see that $h\in L^\infty_x (\R; L^1_y(\R_+))$. Moreover, let $\{f_n\}\subset C_0^\infty (\overline{\R^2_+})$ be an approximation of $f$ in $X_Q$. Then we see 
\begin{align*}
\| h(x,\cdot)-h(x',\cdot)\|_{L^1_y(\R_+)} & = \int_0^\infty | f(x,\Phi (x,y))- f(x',\Phi (x',y))|\, d y\\
& \leq \int_0^\infty | f_n(x,\Phi (x,y))- f_n(x',\Phi (x',y))|\, d y \\
& \quad + \int_0^\infty | f(x,\Phi (x,y))- f_n(x,\Phi (x,y))|\, d y \\
& \quad + \int_0^\infty | f(x',\Phi (x',y))- f_n(x',\Phi (x',y))|\, d y\\
& \leq \int_0^\infty | f_n(x,\Phi (x,y))- f_n(x',\Phi (x',y))|\, d y \\
& \quad + \int_0^\infty | f(x,\eta)- f_n(x,\eta)|\, \frac{d \eta}{(\p_y \Phi) (x, \Phi^{-1}(x,\eta))} \\
& \quad + \int_0^\infty | f(x',\eta)- f_n(x',\eta)|\, \frac{d \eta}{(\p_y\Phi) (x',\Phi^{-1} (x',\eta))}\\
& \leq \int_0^\infty | f_n(x,\Phi (x,y))- f_n(x',\Phi (x',y))|\, d y \\
& \quad + C\| f-f_n \|_{L^\infty_xL^1_y}.
\end{align*}
This shows $h\in BC_x(\R; L^1_y(\R_+))$. Similarly, since $\|h(x,\cdot)\|_{L^1_y(\R_+)} \leq C \| f(x,\cdot) \|_{L^1_y(\R_+)}$, we have $\displaystyle \lim_{|x|\rightarrow\infty} \| h(x,\cdot)\|_{L^1_y(\R_+)} =0$. 
Next we have from 
\begin{align*}
|\p_y h (x,y)| = |\p_y \Phi (x,y) (\p_y f)(x,\Phi(x,y))|\leq C |(\p_y f ) (x,\Phi (x,y))|
\end{align*}
that 
\begin{align*}
\| \p_y h\|_{L^6_xL^2_y}\leq C \| \p_y f\|_{L^6_xL^2_y}<\infty,
\end{align*}
and from 
\begin{align*}
\p_y^2 h (x,y) & = (\p_y\Phi (x,y))^2 (\p_y^2 f)(x,\Phi (x,y)) + \p_y^2\Phi (x,y) (\p_y f) (x,\Phi (x,y)) \\
& = (\p_y\Phi (x,y))^2 (\p_y^2 f)(x,\Phi (x,y)) + \p_y^2\Phi (x,y) \big ( \int_0^{\Phi (x,y)} \p_\tau^2 f (x,\tau) \,d\tau +  (\p_y f) (x,0) \big ), 
\end{align*}
we have 
\begin{align*}
\| \p_y^2 h\|_{L^2(\R^2_+)} \leq C \| \p_y^2f\|_{L^2(\R^2_+)} + C\| \gamma_{\p\R^2_+} \p_y f\|_{L^2(\R)}.
\end{align*}
Here we have used the fact that $\p_y^2 \Phi (x,y) = -\p_y^2 \chi_R (y) Q(x)$ has a compact support in $y$ uniformly in $x$. Finally, we observe that 
 \begin{align*}
 y\p_x h (x,y) &= y\big ( (\p_x f)(x,\Phi (x,y)) - \p_x Q (x) \chi_R (y) (\p_y f) (x,\Phi (x,y)) \big )\\
 & =  y \big ( (\p_x f)(x,\Phi (x,y)) - \p_x Q (x)(\p_y f) (x,\Phi (x,y)) \big ) + y (1- \chi_R (y)) \p_x Q (x)  (\p_y f) (x,\Phi (x,y)).
\end{align*}
This implies, since $|y|= |(1-\frac{\chi_R(y) Q}{y})^{-1} \Phi (x,y)|\leq C \Phi (x,y)$ and since $y(1-\chi_R(y))$ has a compact support in $y$, 
\begin{align*}
\| y\p_x h\|_{L^2(\R^2_+)} \leq C\| y(\p_x f- \p_xQ\,\p_y f)\|_{L^2(\R^2_+)} +  C \| \p_x Q\|_{L^3(\R)} \| \p_y f\|_{L^6_x(\R; L^2_y(\R_+))}<\infty.
\end{align*}
Therefore, $h\in X_Q$ with $Q\equiv 0$ by Lemma \ref{lem.pre:0}. The proof is complete.
\end{proof}

\section{Reduction system}\label{sec.reduction}

One of the difficulties in the analysis of \eqref{ptd:2} is the presence of the term $v\p_y \omega$, which contains the linearly growing term $-y \p_x A\, \p_y \omega$ due to the equality $I_\infty[\omega] =A$. To overcome this difficulty we eliminate this term by the transformation
\begin{align}\label{trans}
w(x,y) = \omega (x, \Phi (x,y)),\quad \Phi (x,y) = y-\chi_R (y) A(x),
\end{align}
where 
\begin{align}\label{chi_R}
\chi_R (y) = \chi (\frac{y}{R\|\p_y\chi\|_{W^{1,\infty}} (1+\|A\|_\infty)})
\end{align}
with $R\geq 2$, and $\chi$ is a smooth cut-off such that $\chi (y)=0$ for $0\leq y\leq 1$ and $\chi(y)=1$ for $y\geq 2$. We may assume that $\p_y\chi\geq 0$. The number $R$ will be taken large enough but independently of data $[\omega,A]$.

\begin{proposition}\label{prop.w} Let $[\omega, A]$ be an $L^2$ strong solution to \eqref{ptd:2}. Then $w\in X_Q$ with $Q\equiv 0$ and $[w,A]$ solves the system
\begin{subequations}\label{ptd:3} 
\begin{align}
& y \p_x w -\p^2_y w = - \theta_1\p_x w - (\theta_{2,1}+\theta_{2,2}+\theta_{2,3}) \p_y w + \theta_3 \p_y^2 w,\quad (x,y)\in \R\times \R_+,\\
& \p_y w|_{y=0} = \p_x|\p_x| A, ~~ I_\infty[w (1- A\p_y \chi_R) ] = A,\quad  x\in \R,
\end{align}
\end{subequations}
where 
\begin{align}
\begin{split}
& \theta_1 = u(\cdot,\Phi)- A \chi_R,\quad (x,y)\in \R\times \R_+,\\
& \theta_{2,1}  =-  \int_0^{\Phi} \int_z^\infty \big (\p_x \omega -\p_xA \p_{\tilde z} \omega \big ) d\tilde z \, dz, \quad (x,y)\in \R\times \R_+,\\
& \theta_{2,2}  =-  \big (1-A\p_y \chi_R \big )^{-1} (1-\chi_R)  \big (u(\cdot,\Phi) - \Phi \big ) \p_x A, \quad (x,y)\in \R\times \R_+,\\
& \theta_{2,3} =  (1-A\p_y\chi_R)^{-3} \big (-2A\p_y\chi_R+(A\p_y\chi_R)^2\big ) A \p_y^2\chi_R  -  (1-A\p_y \chi_R)^{-1} A \p_y^2\chi_R, \\
& \qquad \qquad \quad ~~~~~~~~~~~~~~~~~~~~~~~~~~~~~~~~~~~~~~~~~~~~~~~~~~(x,y)\in \R\times \R_+,\\
& \theta_{3}  = (1-A\p_y\chi_R)^{-2} \big (-2A\p_y\chi_R+(A\p_y\chi_R)^2\big ), \quad (x,y)\in \R\times \R_+.
\end{split}
\end{align}
Here \eqref{ptd:3} is formulated in the similar manner as Definition \ref{def.sol} (ii).
\end{proposition}

The proof of Proposition \ref{prop.w} is postponed to the appendix.
The key point of \eqref{ptd:3} is that the linearly growing term $-y\p_xA\, \p_y \omega$ in the original system \eqref{ptd:2} is replaced mainly by $-\theta_{2,1} \p_y w$ in \eqref{ptd:3}, for which the growth in $y$ is reduced.

Let $[\omega, A]$ be an $L^2$ strong solution to \eqref{ptd:2}. Then $w\in X_Q$ with $Q\equiv 0$ by Proposition \ref{prop.pre:2}. Next we take $w_b\in X_Q$ with $Q\equiv 0$ as the solution to \eqref{eq.b} with $g=\p_x |\p_x| A\in H^\frac16(\R) \cap \dot{H}^{-1}(\R)$. Then we set 
\begin{align}
w_{in} =w-w_b,
\end{align}
which belongs to $X_Q$ with $Q\equiv 0$, and solves the equation 
\begin{subequations}\label{ptd:4} 
\begin{align}
& y \p_x w_{in} -\p^2_y w_{in} = - \theta_1\p_x w - (\theta_{2,1}+\theta_{2,2}+\theta_{2,3}) \p_y w + \theta_3 \p_y^2 w,\quad (x,y)\in \R\times \R_+,\\
& \p_y w_{in}|_{y=0} = 0, \quad x\in\R,\\
& I_\infty[w_{in}] = -I_\infty[w_b] +  \big ( 1+ I_\infty [w\p_y \chi_R]\big ) A,\quad  x\in \R.
\end{align}
\end{subequations}
Here \eqref{ptd:4} is again formulated in the similar manner as Definition \ref{def.sol} (ii).
As we will see in the next subsection, the third equation of \eqref{ptd:4}  in fact defines the elliptic equation for $A$, of the order $|\p_x|^\frac43$. 
For later use, in this section we set 
\begin{align}\label{df.sc}
\begin{split}
\|[\omega,A]\|_{sc} & = \| \omega \|_{L_x^\infty (\R; L^1_y (\R_+))} \| \p_x A\|_{L^3(\R)} + \|\omega\|_{L^\infty(\R^2_+)} \\
& \quad + \| \p_y \omega\|_{L^6_x (\R; L^2_y (\R_+))} + \| y( \p_x \omega - \p_x A\, \p_y \omega )\|_{L^{\frac{12}{5},\infty}_x(\R; L^{\frac43,\infty}_y(\R_+))}.
\end{split}
\end{align}

\subsection{Elliptic equation for $A$}\label{subsec.A}

In this subsection we focus on the equation 
\begin{align}\label{eq.A}
-I_\infty[w_b] +  \big ( 1+ I_\infty [w\p_y \chi_R]\big ) A= I_\infty[w_{in}], \quad x\in \R,
\end{align}
which appears in the third equation of \eqref{ptd:4}.

\begin{lemma}\label{lem.I} Let $f\in X_Q$ with $Q\equiv 0$. Then 
\begin{align}\label{est.lem.I}
\| |\p_x|^\frac56 I_\infty[f]\|_{L^2(\R)} \leq C \big ( \|y\p_x f\|_{L^2 (\R^2_+)} + \| |\p_x|^\frac23 f\|_{L^2(\R^2_+)}\big ).
\end{align}
\end{lemma}

\begin{proof} It suffices to consider the case $f\in C_0^\infty (\overline{\R^2_+})$. Let $\hat{f}(\xi,y)$ be the Fourier transform of $f$ in the $x$ variable. Then 
\begin{align*}
|I_\infty[\hat{f}] (\xi)| & \leq \int_0^\infty |\hat{f}(\xi,y)| \, dy \\
& \leq \int_0^{|\xi|^{-\frac13}} |\hat{f}(\xi,y)| \, dy + \int_{|\xi|^{-\frac13}}^\infty \frac{1}{y} \, y|\hat{f}(\xi,y)|\, d y\\
& \leq C |\xi|^{-\frac16} \| \hat{f}(\xi,\cdot)\|_{L^2_y (\R_+)} + C |\xi|^{\frac16} \| y\hat{f} (\xi,\cdot)\|_{L^2_y(\R_+)},
\end{align*}
which gives $\||\xi|^\frac56 I_\infty[\hat{f}]\|_{L^2_\xi(\R)} \leq C \| |\xi|^\frac23 \hat{f} \|_{L^2_\xi(\R;  L^2_y(\R_+))} + C \| |\xi| y \hat{f}\|_{L^2_\xi (\R; L^2_y (\R_+))}$. Hence, \eqref{est.lem.I} follows from the Plancherel theorem. The proof is complete. 
\end{proof}

The main result of this subsection is stated as follows.
\begin{proposition}\label{prop.A} There exist $R_0\geq 2$ and $\delta_0>0$ such that if $R\geq R_0$ and $\| \omega\|_{L^\infty (\R^2_+)}\leq \delta_0$, then
\begin{align}\label{est.prop.A}
\| |\p_x|^\frac56 A \|_{H^\frac43 (\R)} \leq C\big ( R^{-1} \| w_{in} \|_Y +  \| |\p_x|^\frac56 I_\infty [w_{in}]\|_{L^2(\R)}\big ).
\end{align}
In particular, we have $\| |\p_x|^\frac56 A \|_{H^\frac43 (\R)} \leq C \| w_{in}\|_Y$ by Lemma \ref{lem.I}.
\end{proposition}

\begin{proof} Let us first compute $I_\infty[w_b]$. For this purpose let $\hat{w}_b$ be the Fourier transform of $w_b$ in the $x$ variable. Then, the proof of Proposition \ref{prop.pre:1} implies the formula 
\begin{align}%\label{proof.prop.A:1}
I_\infty[\hat{w}_b] (\xi) & = 
\begin{cases}
& \displaystyle \frac{1}{{\rm Ai}'(0) e^{\frac{\pi}{6}i} \xi^\frac13} \int_0^\infty {\rm Ai} (e^{\frac{\pi}{6}i} \xi^\frac13 y) \, dy \, i\xi|\xi| \hat{A}(\xi),\quad \xi>0, \\
& \displaystyle \frac{1}{{\rm Ai}'(0) e^{-\frac{\pi}{6}i} (-\xi)^\frac13} \int_0^\infty {\rm Ai} (e^{-\frac{\pi}{6}i} (-\xi)^\frac13 y) \, dy \, i\xi|\xi| \hat{A}(\xi),\quad \xi<0, 
\end{cases} \nonumber \\
& =
\begin{cases}
& \displaystyle \frac{1}{{\rm Ai}'(0) e^{\frac{\pi}{3}i} \xi^\frac23} \int_0^\infty {\rm Ai} (r) \, dr \, i\xi|\xi| \hat{A}(\xi),\quad \xi>0, \\
& \displaystyle \frac{1}{{\rm Ai}'(0) e^{-\frac{\pi}{3}i} (-\xi)^\frac23} \int_0^\infty {\rm Ai} (r) \, dr \, i\xi|\xi| \hat{A}(\xi),\quad \xi<0. 
\end{cases}\nonumber 
\end{align}
Since $\displaystyle \int_0^\infty {\rm Ai} (r) \, dr=\frac13$, we have 
\begin{align}\label{proof.prop.A:2}
I_\infty [\hat{w}_b] (\xi ) = \sigma (\xi) i\xi |\xi|^\frac13 \hat{A}(\xi),\quad \sigma (\xi) =
\begin{cases} 
& \displaystyle \frac{1}{3{\rm Ai}'(0)e^{\frac{\pi}{3}i}},\quad \xi>0,\\
& \displaystyle \frac{1}{3{\rm Ai}'(0)e^{-\frac{\pi}{3}i}},\quad \xi<0.
\end{cases}
\end{align}
We note that 
\begin{align}\label{proof.prop.A:3}
\Re (\sigma (\xi)) = \frac{1}{6{\rm Ai}'(0)},\quad \xi\ne 0.
\end{align}
Therefore, the multiplier
\begin{align}\label{proof.prop.A:4}
m(\xi) = 1- \sigma (\xi) i\xi |\xi|^\frac13
\end{align}
is invertible for any $\xi\in \R$, and has the estimate 
\begin{align}\label{proof.prop.A:5}
|\frac{1}{m(\xi)}| \leq \frac{C}{1+|\xi|^\frac43},\quad \xi\in \R.
\end{align}
Hence, \eqref{eq.A} is written as 
\begin{align}\label{proof.prop.A:6}
A = - m(-i\p_x)^{-1} I_\infty [w\p_y\chi_R] A + m(-\p_x)^{-1} I_\infty [w_{in}].
\end{align}
Here $m(-i\p_x)^{-1}$ is the Fourier multiplier operator whose symbol is $1/m(\xi)$.
Since $m(-i\p_x)^{-1}$ is bounded from $L^2(\R)$ to $H^\frac43(\R)$ and since $|\p_x|^\frac56$ commutes with $m(-i\p_x)^{-1}$, we have  from \eqref{proof.prop.A:6},
\begin{align}\label{proof.prop.A:7}
\| |\p_x|^\frac56 A \|_{H^\frac43(\R)} &\leq C\| |\p_x|^\frac56 \big ( I_\infty [w\p_y\chi_R] A\big ) \|_{L^2(\R)} + C \| |\p_x|^\frac56 I_\infty[w_{in}]\|_{L^2(\R)} \nonumber \\
\begin{split}
& \leq C \| |\p_x|^\frac56 I_{\infty}[w\p_y\chi_R]\|_{L^2(\R)} \|A\|_{L^\infty (\R)} + C \| I_\infty[w\p_y\chi_R]\|_{L^\infty(\R)} \| |\p_x|^\frac56 A\|_{L^2(\R)}\\
& \quad + C \| |\p_x|^\frac56 I_\infty[w_{in}]\|_{L^2(\R)}.
\end{split}
\end{align}
Here we have used the Kato-Ponce inequality in the homogeneous Sobolev spaces; see, e.g., Grafakos and Oh \cite{GraOh}.
Lemma \ref{lem.I} implies 
\begin{align}\label{proof.prop.A:8}
\| |\p_x|^\frac56 I_{\infty}[w\p_y\chi_R]\|_{L^2(\R)} &\leq C \big ( \| y\p_x w\p_y \chi_R\|_{L^2(\R^2_+)} + \| |\p_x|^\frac23 w\p_y \chi_R \|_{L^2(\R^2_+)}\big )\nonumber \\
& \leq C \big ( \| y\p_x w\|_{L^2(\R^2_+)} + \| |\p_x|^\frac23 w \|_{L^2(\R^2_+)}\big ) \| \p_y \chi_R \|_{L^\infty_y} \nonumber \\
& \leq C \big ( \| y\p_x w\|_{L^2(\R^2_+)} + \| |\p_x|^\frac23 w \|_{L^2(\R^2_+)}\big ) R^{-1} (1+\|A\|_{L^\infty(\R)})^{-1}.
\end{align}
In the last line above, we have used the definition of $\chi_R$ in \eqref{chi_R}.
On the other hand, since $\|w\|_{L^\infty(\R^2_+)}=\|\omega\|_{L^\infty (\R^2_+)}$, we have 
\begin{align}\label{proof.prop.A:9}
\| I_\infty [w\p_y\chi_R]\|_{L^\infty (\R)} \leq \|\p_y\chi_R\|_{L^1_y(\R_+)} \| w\|_{L^\infty(\R^2_+)} \leq \|\omega\|_{L^\infty (\R^2_+)}.
\end{align}
Hence, \eqref{proof.prop.A:7}, \eqref{proof.prop.A:8}, and \eqref{proof.prop.A:9} yield
\begin{align}\label{proof.prop.A:10}
\| |\p_x|^\frac56 A \|_{H^\frac43(\R)} 
& \leq CR^{-1}  \big ( \| y\p_x w\|_{L^2(\R^2_+)} + \| |\p_x|^\frac23 w \|_{L^2(\R^2_+)}\big ) + C \|\omega\|_{L^\infty (\R^2_+)} \| |\p_x|^\frac56 A\|_{L^2(\R)} \nonumber \\
& \quad + C \| |\p_x|^\frac56 I_\infty[w_{in}]\|_{L^2(\R)}\nonumber \\
\begin{split}
& \leq CR^{-1}  \big ( \| y\p_x w_b\|_{L^2(\R^2_+)} +\| y\p_x w_{in} \|_{L^2(\R^2_+)} + \| |\p_x|^\frac23 w_b \|_{L^2(\R^2_+)} + \| |\p_x|^\frac23 w_{in} \|_{L^2(\R^2_+)}\big ) \\
& \quad + C \|\omega\|_{L^\infty (\R^2_+)} \| |\p_x|^\frac56 A\|_{L^2(\R)} + C \| |\p_x|^\frac56 I_\infty[w_{in}]\|_{L^2(\R)}.
\end{split}
\end{align}
By applying Proposition \ref{prop.pre:1}, we have $ \| y\p_x w_b\|_{L^2(\R^2_+)}  + \| |\p_x|^\frac23 w_b \|_{L^2(\R^2_+)} \leq C \| |\p_x|^\frac16 \p_x|\p_x| A\|_{L^2(\R)}\leq \||\p_x|^\frac56 A\|_{H^\frac43(\R)}$. Hence, \eqref{proof.prop.A:10} implies that there exist $R_0\geq 2$ and $\delta_0>0$ such that if $R\geq R_0$ and $\| \omega\|_{L^\infty (\R^2_+)}\leq \delta_0$ then \eqref{est.prop.A} holds. The proof is complete.
\end{proof}

\subsection{Estimate of $w_{in}$}

In this subsection we establish the estimate of $w_{in}\in X_Q$ with $Q\equiv 0$ satisfying the first two equations of \eqref{ptd:4}. Recall that $w=w_{in} + w_b$, and $w$ and $\omega$ are related as in \eqref{trans}. 
\begin{lemma}\label{lem.in:1} Let $[\omega,A]$ be an $L^2$ strong solution to \eqref{ptd:3} and $\omega\in L^\infty (\R^2_+)$. Then 
\begin{align}\label{est.lem.in:1:1}
\begin{split}
& \| \theta_1 \p_x w \|_{L^2(\R^2_+)} + \sum_{j=1,2}\|\theta_{2,j} \p_y w\|_{L^2(\R^2_+)} + \| \theta_3 \p_y^2 w\|_{L^2(\R^2_+)} \\
& \leq C (R \|[\omega,A]\|_{sc} + R^{-1} ) \| w\|_Y + C R  \|\p_x A\|_{L^3_x(\R)} \| \p_y \omega \|_{L^6_x(\R; L^2_y (\R_+))},
\end{split}
\end{align}
and 
\begin{align}\label{est.lem.in:1:2}
\| \theta_{2,3} \p_y w \|_{L^2(\R_+)} & \leq C R^{-1} \big ( \| \p_y^2 w_{in} \|_{L^2(\R^2_+)} + R^{-\frac12} \|\p_y w_b \|_{L^\infty_y (\R_+; L_x^2(\R))}\big ). 
\end{align}
\end{lemma}

\begin{proof} We first observe from \eqref{chi_R} that $\frac{|A\chi_R(y)|}{y}| \leq CR^{-1}$ with a numerical constant $C>0$. This implies from \eqref{trans} that 
\begin{align*}
|\Phi(x, y)| \le y + \chi_R(y) |A(x)| \le (1 + C R^{-1})y, 
\end{align*}
 and therefore that 
\begin{align*}
|u(x,\Phi (x,y))| + |A(x) \chi_R (y)|\leq \|\omega\|_{L^\infty(\R^2_+)} \Phi (x,y)  + CR^{-1}y\leq C y \big (\|\omega\|_{L^\infty (\R^2_+)} + R^{-1}\big ).
\end{align*}
This gives 
\begin{align}\label{proof.lem.in:1:1}
\|\theta_1\p_x w\|_{L^2(\R^2_+)}\leq C \big (\|\omega\|_{L^\infty (\R^2_+)} + R^{-1}\big )\| y\p_x w\|_{L^2(\R^2_+)}.
\end{align}
Next we see from Proposition \ref{prop.pre.product} that 
\begin{align}%\label{proof.lem.in:1:2}
|\theta_{2,1} (x,y)| & \leq C \int_0^{\Phi (x,y)} \| \frac{1}{\tilde z} \|_{L_{\tilde z}^{4,1}(z,\infty)} \| \tilde z (\p_x\omega-\p_x A\, \p_{\tilde z} \omega)(x,\cdot)\|_{L^{\frac43,\infty}_{\tilde z}(\R_+)} \, d z \nonumber \\
& \leq C \int_0^{\Phi (x,y)} z^{-\frac34} \, d z\, \| \tilde z (\p_x\omega-\p_x A\, \p_{\tilde z} \omega )(x,\cdot)\|_{L^{\frac43,\infty}_{\tilde z}(\R_+)} \nonumber \\
& \leq C y^\frac14 \| \tilde z (\p_x\omega-\p_x A\, \p_{\tilde z} \omega )(x,\cdot)\|_{L^{\frac43,\infty}_{\tilde z}(\R_+)}.\nonumber 
\end{align}
Then, again from Proposition \ref{prop.pre.product},
\begin{align*} 
\| \theta_{2,1}\p_y w\|_{L^2(\R^2_+)}\leq C \| \tilde z (\p_x\omega-\p_x A\, \p_{\tilde z} \omega )(x,\cdot)\|_{L^{\frac{12}{5},\infty}_x (\R; L^{\frac43,\infty}_{\tilde z}(\R_+)} \| y^\frac14 \p_y w\|_{L^{12,2}_x (\R; L^2_y(\R_+))},
\end{align*}
which gives from Corollary \ref{cor:lem.pre:1} and \eqref{est.lem.pre:1:1} of Proposition \ref{lem.pre:1},
\begin{align}\label{proof.lem.in:1:3}
\|\theta_{2,1}\p_y w\|_{L^2(\R^2_+)}\leq C \|[\omega,A]\|_{sc} \| w\|_Y.
\end{align}
Next we observe from $\|u\|_{L^\infty (\R^2_+)} + \|A\|_{L^\infty(\R)} \leq 2 \|\omega\|_{L^\infty_x(\R; L^1_y(\R_+))}$ and $\Phi (x,y)(1-\chi_R(y)) \leq Cy(1-\chi_R (y))\leq C R (1+\|A\|_{L^\infty(\R)})$ that 
\begin{align*}
|\theta_{2,2} (x,y)| & \leq C \|\omega\|_{L^\infty_x(\R; L^1_y(\R_+))} |\p_x A(x)|\\
& \quad + C (1-\chi_R (y))  \Phi (x,y) |\p_x A(x)|\\
& \leq C R \|\omega\|_{L^\infty_x(\R; L^1_y(\R_+))} |\p_x A(x)|  + C R   |\p_x A(x)|,
\end{align*}
which implies
\begin{align}\label{proof.lem.in:1:4}
\|\theta_{2,2} \p_y w\|_{L^2(\R^2_+)} & \leq CR \|\omega\|_{L^\infty_x(\R;L^1_y(\R_+))}  \|\p_x A\|_{L^3_x(\R)} \| \p_y w\|_{L^6_x(\R; L^2_y (\R_+))}  \nonumber \\
& \quad + CR \|\p_x A\|_{L^3_x(\R)} \| \p_y w\|_{L^6_x(\R; L^2_y (\R_+))} \nonumber \\
& \leq  C R \|[\omega,A]\|_{sc} \|w\|_Y + C R  \|\p_x A\|_{L^3_x(\R)} \| \p_y\omega\|_{L^6_x(\R; L^2_y (\R_+))}.
\end{align}
Here we have used $\| \p_y w\|_{L^6_x(\R; L^2_y (\R_+))}  \leq C\|w\|_Y$, which follows from Corollary \ref{cor:lem.pre:1} and \eqref{est.lem.pre:1:1} of Proposition \ref{lem.pre:1}.
Next we have from $\|A y \p_y^2\chi_R\|_{L^\infty(\R^2_+)} \leq CR^{-1}$ and $\| A\p_y^2\chi_R\|_{L^2_y(\R_+; L^\infty_x(\R))}\leq CR^{-\frac32}$, 
\begin{align}\label{proof.lem.in:1:5}
\|\theta_{2,3}\p_y w\|_{L^2(\R^2_+)} & \leq C \| A\p_y^2\chi_R \p_y w\|_{L^2(\R^2_+)}  \nonumber \\
& \leq C\| A\p_y^2 \chi_R \p_y w_{in} \|_{L^2(\R^2_+)} + C \| A\p_y^2 \chi_R \p_y w_b\|_{L^2(\R^2_+)} \nonumber \\
& \leq C \| A y \p_y^2 \chi_R \|_{L^\infty (\R^2_+)} \| y^{-1} \p_y w_{in}\|_{L^2(\R^2_+)} \nonumber \\
& \quad + C \| A\p_y^2\chi_R \|_{L^2_y(\R_+; L^\infty_x(\R))}  \|\p_y w_b \|_{L^\infty_y (\R_+; L^2_x(\R))}\nonumber \\
& \leq CR^{-1} \|\p_y^2 w_{in}\|_{L^2(\R^2_+)} + C R^{-\frac32} \|\p_y w_b \|_{L^\infty_y (\R_+; L^2_x(\R))}.
\end{align} 
Here we have used the Hardy inequality in the last line.
Finally, since $\|\theta_3\|_{L^\infty (\R^2_+)} \leq CR^{-1}$, we have 
\begin{align}\label{proof.lem.in:1:6}
\| \theta_3\p_y^2 w\|_{L^2(\R^2_+)}\leq CR^{-1} \|\p_y^2w \|_{L^2(\R^2_+)}.
\end{align}
Collecting \eqref{proof.lem.in:1:1}, \eqref{proof.lem.in:1:3}, \eqref{proof.lem.in:1:4}, \eqref{proof.lem.in:1:5}, and \eqref{proof.lem.in:1:6}, we obtain \eqref{est.lem.in:1:1} and \eqref{est.lem.in:1:2}. The proof is complete.
\end{proof}

The main result of this subsection is as follows.
\begin{proposition}\label{prop.in:1} Let $[\omega,A]$ be an $L^2$ strong solution to \eqref{ptd:3} and $\omega\in L^\infty (\R^2_+)$. There exist $R_1\geq R_0$ and $\delta_1\in (0,\delta_0]$ such that if $R\geq R_1$ and $R\|[\omega,A]\|_{sc}\leq \delta_1$, then 
\begin{align}\label{est.prop.in:1}
\| w_{in}\|_Y\leq C (R\| [\omega, A]\|_{sc} +R^{-1}) \| w_b \|_Y + C R^{-\frac32} \| \p_y w_b \|_{L^\infty_y(\R_+; L^2_x(\R))}.
\end{align}
\end{proposition}

\begin{proof} Since $w_{in}\in X_Q$ with $Q\equiv 0$ satisfies \eqref{ptd:4}, Proposition \ref{lem.pre:1} and Lemma \ref{lem.in:1} imply that 
\begin{align*}
\| w_{in}\|_Y&\leq C \big ( \| \theta_1 \p_x w \|_{L^2(\R^2_+)} + \sum_{j=1,2,3}\|\theta_{2,j} \p_y w\|_{L^2(\R^2_+)} + \| \theta_3 \p_y^2 w\|_{L^2(\R^2_+)}\big ) \\
& \leq  C (R\|[\omega,A]\|_{sc} + R^{-1} ) \| w\|_Y + C R  \|\p_x A\|_{L^3_x(\R)} \| \p_y \omega \|_{L^6_x(\R; L^2_y (\R_+))} \\
& \quad + C R^{-1} \big ( \| \p_y^2 w_{in} \|_{L^2(\R^2_+)} + R^{-\frac12} \|\p_y w_b \|_{L^\infty_y (\R_+; L_x^2(\R))}\big )  \\
\begin{split}
& \leq C (R \|[\omega,A]\|_{sc} + R^{-1} ) \big ( \| w_{in} \|_Y + \|w_b \|_{Y}\big ) + C R  \|w_{in} \|_Y \| \p_y \omega \|_{L^6_x(\R; L^2_y (\R_+))} \\
& \quad + C R^{-\frac32} \|\p_y w_b \|_{L^\infty_y (\R_+; L_x^2(\R))}
\end{split}
\end{align*}
Here, in the last line, we have used $\|\p_x A\|_{L^3(\R)} \leq C \| |\p_x|^\frac56 A\|_{H^\frac43(\R)} \leq C\|w_{in}\|_Y$ by the Sobolev embedding inequality and Proposition \ref{prop.A}.
Hence,
\begin{align}
\|w_{in}\|_Y & \leq C (R \|[\omega,A]\|_{sc} + R^{-1} ) \big ( \| w_{in} \|_Y + \|w_b \|_{Y}\big ) + C R^{-\frac32} \|\p_y w_b \|_{L^\infty_y (\R_+; L_x^2(\R))}.
\end{align} 
Here $C>0$ is a numerical constant. Therefore, there exist $R_1\geq R_0$ and $\delta_1\in (0,\delta_0]$ such that, if $R\geq R_1$ and $R\|[\omega,A]\|_{sc}\leq \delta_1$, then \eqref{est.prop.in:1} holds. The proof is complete.
\end{proof}

\section{Proof of Theorem \ref{main.thm:1}}\label{sec.proof.main}
Let us recall that $w_b$ is the solution to \eqref{eq.b} with $g=\p_x|\p_x|A\in H^\frac16(\R)\cap \dot{H}^{-1}(\R)$. Hence, Propositions \ref{prop.in:1} and \ref{prop.pre:1} imply
\begin{align*}
\|w_{in}\|_Y & \leq C (R \| [\omega, A]\|_{sc} +R^{-1}) \| |\p_x|^\frac16 \p_x|\p_x|A\|_{L^2(\R)} + CR^{-\frac32} \| \p_x|\p_x|A\|_{L^2(\R)} \\
& \leq C (R \| [\omega, A]\|_{sc} + R^{-1}) \| |\p_x|^\frac56 A\|_{H^\frac43(\R)},
\end{align*}
which yields from Proposition \ref{prop.A} that 
\begin{align}\label{proof.est.in}
\| w_{in}\|_Y\leq C (R \| [\omega, A]\|_{sc} + R^{-1}) \|w_{in}\|_Y.
\end{align}
Here $C>0$ is a numerical constant. Let us fix $R\geq R_1$ large enough so that $CR^{-1}\leq 1/4$ in \eqref{proof.est.in}. Then, if $CR\|[\omega,A]\|_{sc}\leq 1/4$ for such a fixed $R$, then $w_{in} = 0$. This also implies $[w_b,A] = [0,0]$ by Propositions \ref{prop.A} and \ref{prop.pre:1}. Thus, we obtain $w = 0$, i.e., $[\omega,A] =[0,0]$. The proof of Theorem \ref{main.thm:1} is complete.

\appendix
 
\section{Proof of Lemma \ref{lem.pre:0}}

It is straightforward that the density argument implies \eqref{est.lem.pre:0} for any $f\in \overline{C_0^\infty (\overline{\R^2_+})}^{\|\cdot \|_{L^\infty_xL^1_y}}$. 
Conversely, suppose that $f\in BC_x (\R; L^1_y (\R_+))$, $\p_y f\in L^6_x(\R;L^2_y(\R_+))$, $\p_y^2 f, \, y(\p_x f-\p_xQ\,\p_yf)\in L^2(\R^2_+)$, and \eqref{est.lem.pre:0} holds.
In particular, $f (x,\cdot)$ is bounded and uniformly continuous in $x\in \R$ with valued in $L^1_y(\R_+)$. We also note that $|\p_y f (x,y)|^2\leq 2\| \p_y^2 f(x,\cdot)\|_{L^2_y} \|\p_y f(x,\cdot)\|_{L^2_y}$ implies $\p_y f\in L_x^3(\R; L^\infty_y(\R_+))$ and $\gamma_{\p\R^2_+} \p_y f\in L_x^3(\R)$. Moreover, $|f(x,y)|^2\leq 2\|\p_y f(x,\cdot)\|_{L^\infty_y} \| f(x,\cdot)\|_{L^1_y}$ implies that $f\in L^6_x (\R; L^\infty_y (\R_+))$, and then, together with $f\in L^\infty_x(\R;L^1_y(\R_+))$, we have $f\in L^{12}_x(\R; L^2_y (\R_+))$. Let  us introduce the extention
\begin{align*}
\tilde f (x,y) = 
\begin{cases}
& f (x,y), \quad y>0,\\
& \displaystyle \sum_{j=1}^3 (-j)^k \lambda_j f (x,-jy), \quad y<0,
\end{cases}
\qquad \sum_{j=1}^3 (-j)^k \lambda_j=1, ~~k=0,1,2.
\end{align*}
Then, $\tilde f\in BC_x (\R; L^1_y (\R))$,  $\p_y \tilde f\in L^6_x (\R; L^2_y(\R))$,  $y\p_x f, \, \p_y^2 \tilde f \in L^2(\R^2)$, and $\displaystyle \lim_{|x|\rightarrow \infty} \|\tilde f(x,\cdot)\|_{L^1_y(\R)}=0$ holds. However, we do not have $y(\p_x \tilde f- \p_x Q\, \p_y \tilde f)\in L^2(\R^2)$ in general,  but
\begin{align*}
& y (\p_x \tilde f - \p_x Q(x) \, \p_y \tilde f ) \\
& = y \sum_{j=1}^3 \lambda_j  \p_x f (x,-jy) - y \p_x Q (x)\, \sum_{j=1}^3 (-j) \lambda_j (\p_y f) (x,-j y)\\
& =\sum_{j=1}^3 \lambda_j  y \big (\p_x f (x,-jy) -\p_x Q(x) \,(\p_y f) (x,-jy)\big ) +\sum_{j=1}^3 y\p_x Q (x) \lambda_j  (1-j)  (\p_y f) (x,-j y) 
\end{align*}
for $y<0$ implies that $y(\p_x \tilde f- \p_x Q\, \p_y \tilde f)\in L_x^2(\R; L^2_y (\{y>-4\}))$.
Then we set 
 \begin{align*}
\tilde f_n & =   j_{\epsilon_n} *_x j_{\epsilon_n} *_y (\psi_{1,n} \psi_{2,m} \tilde f)~~~{\rm with}~m=m(n),
\end{align*}
where $j_\epsilon$ is the one-dimensional Friedrichs mollifier, and $\psi_{1,n} (x) =\psi (\frac{x}{n})$and $\psi_{2,m} (y) = \psi (\frac{y}{m})$ are smooth cut-off functions with $\psi_j (r)=1$ for $|r|\leq 1$ and $\psi_j (r)=0$ for $|r|\geq 2$. We will take $m=m(n)=n^\frac{5}{12}$, while $\epsilon_n$ is chosen small enough so that \eqref{ep_n} below is satisfied. We set $f_n$ as the restriction of $\tilde f_n$ to $\overline{\R^2_+}$. It is clear that $f_n\in C_0^\infty (\overline{\R^2_+})$. 
Moreover, by using the fact that $\tilde f$ is unifromly continuous in $x\in \R$ with valued in $L^1_y(\R)$ and  $\displaystyle \lim_{|x|\rightarrow \infty} \|\tilde f(x,\cdot)\|_{L^1_y(\R)}=0$, we have
\begin{align*}
\| f-f_n \|_{L^\infty_x(\R; L^1_y (\R_+))} \leq \| \tilde f-\tilde f_n\|_{L^\infty_xL^1_y}&\leq \|\tilde f-j_{\epsilon_n}*_x j_{\epsilon_n}*_y \tilde f\|_{L^\infty_x L^1_y} + \| (1-\psi_{1,n} \psi_{2,m(n)}) \tilde f\|_{L^\infty_xL^1_y} \\
&\rightarrow 0 \quad (n\rightarrow \infty),
\end{align*}
and 
\begin{align*}
\|\p_y f-\p_y f_n \|_{L^6_x(\R; L_y^2(\R_+))} &\leq \|\p_y \tilde f- \p_y \tilde f_n\|_{L^6_x L^2_y} \\
& \leq \| \p_y \tilde f- j_{\epsilon_n} *_x j_{\epsilon_n} *_y (\psi_{1,n}\psi_{2,m(n)} \p_y \tilde f )\|_{L^6_xL_y^2} + \|  \psi_{1,n} \p_y \psi_{2,m(n)}  \tilde f\|_{L^6_xL_y^2} \\
& \leq \| \p_y \tilde f- j_{\epsilon_n} *_x j_{\epsilon_n} *_y (\psi_{1,n}\psi_{2,m(n)} \p_y \tilde f) \|_{L^6_xL_y^2} + C m(n)^{-1} n^\frac{1}{12} \| \tilde f\|_{L^{12}_xL^2_y}\\
& \rightarrow 0 \quad (n\rightarrow \infty),
\end{align*}
and similarly, 
\begin{align*}
\| \p_y^2  f-\p_y^2 f_n \|_{L^2} & \leq \|\p_y^2 \tilde f-\p_y^2\tilde f_n\|_{L^2}\\
& \leq \| \p_y^2 \tilde f- j_{\epsilon_n} *_x j_{\epsilon_n} *_y (\psi_{1,n}\psi_{2,m(n)} \p_y^2 \tilde f)\|_{L^2} \\
& \quad + 2 \|  \psi_{1,n} \p_y \psi_{2,m(n)}  \p_y f\|_{L^2} + \|  \psi_{1,n} \p_y^2 \psi_{2,m(n)}  \tilde f\|_{L^2} \\
& \leq \| \p_y^2 f- j_{\epsilon_n} *_x j_{\epsilon_n} *_y  (\psi_{1,n}\psi_{2,m(n)} \p_y^2 \tilde f) \|_{L^2} \\
& + Cm(n)^{-1} n^\frac13 \| \p_y f\|_{L^6_xL^2_y}  + C m(n)^{-2} n^\frac{5}{12} \| \tilde f\|_{L^{12}_xL_y^2}\\
& \rightarrow 0 \quad (n\rightarrow \infty).
\end{align*}
Next we set $h_n =  \psi_{1,n} \psi_{2,m(n)} \tilde f$ and write $S=y(\p_x - \p_x Q\,\p_y)$ to simplify the notation. Then,
\begin{align*}
\| S f - S f_n \|_{L^2(\R^2_+)} & \leq \| S \big (1- \psi_{1,n} \psi_{2,m(n)} \big ) f \|_{L^2(\R^2_+)}  + \| Sh_n - S   j_{\epsilon_n} *_x  j_{\epsilon_n} *_y  h_n \|_{L^2(\R^2_+)}  \\
& \leq \|  \big (1- \psi_{1,n} \psi_{2,m(n)} \big ) S f \|_{L^2(\R^2_+)} + \| y\p_x \psi_{1,n} \psi_{2,m(n)}  f \|_{L^2(\R^2_+)}\\
& \quad + \| y \p_x Q\, \psi_{1,n} \p_y \psi_{2,m(n)} f \|_{L^2(\R^2_+)} + \| Sh_n - S   j_{\epsilon_n} *_x  j_{\epsilon_n} *_y  h_n \|_{L^2(\R^2_+)}\\
& \leq  \|  \big (1- \psi_{1,n} \psi_{2,m(n)} \big ) S f \|_{L^2(\R^2_+)} + C m(n) n^{-1} n^\frac{5}{12}\| f \|_{L^{12}_xL^2_y}\\
& \quad + C \| \p_x Q\|_{L^\frac{12}{5}_x} \| f \|_{L^{12}_x(\R; L^2_y (\{m(n)\leq y\leq 2 m(n)\}))} + \| Sh_n - S   j_{\epsilon_n} *_x  j_{\epsilon_n} *_y  h_n \|_{L^2(\R^2_+)}.
\end{align*}
It is clear that the first three terms in tha last line tend to $0$ as $n\rightarrow \infty$.
Let us consider the term $\| Sh_n - S   j_{\epsilon_n} *_x  j_{\epsilon_n} *_y  h_n \|_{L^2(\R^2_+)}$.
We see 
\begin{align}
&y(\p_x- \p_x Q (x)\, \p_y ) (h_n -  j_{\epsilon_n} *_x  j_{\epsilon_n} *_y  h_n) \nonumber \\
& =y(\p_x- \p_x Q(x)\, \p_y ) \int_{\R^2} \big ( h_n (x,y)  - h_n (x-\epsilon_n\tilde x, y-\epsilon_n\tilde y) \big ) \, j(\tilde x) j(\tilde y) \, d\tilde x d\tilde y  \nonumber \\
& = \int_{\R^2} \big ( S h_n (x,y)  - (S h_n) (x-\epsilon_n \tilde x, y-\epsilon_n \tilde y) \big ) \, j(\tilde x) j(\tilde y) \, d\tilde x d\tilde y  \nonumber \\
& \quad + \int_{\R^2}  \tilde y h_n (x-\epsilon_n\tilde x, y-\epsilon_n \tilde y) \, \p_{\tilde x} j(\tilde x) j(\tilde y) \, d\tilde x d\tilde y  \nonumber \\
& \quad + \int_{\R^2} \big ( (y-\epsilon_n\tilde y) \p_x Q(x-\epsilon_n \tilde x) - y\p_x Q(x)\big ) \p_y h_n (x-\epsilon_n\tilde x, y-\epsilon_n\tilde y) \, j(\tilde x) j(\tilde y) \, d\tilde x d\tilde y\nonumber \\
& = \int_{\R^2} \big ( S h_n (x,y)  - (S h_n) (x-\epsilon_n \tilde x, y-\epsilon_n\tilde y) \big ) \, j(\tilde x) j(\tilde y) \, d\tilde x d\tilde y  \nonumber \\
& \quad + \int_{\R^2}   \big (h_n (x-\epsilon_n \tilde x, y-\epsilon_n\tilde y) - h_n (x-\epsilon_n\tilde x, y) \big ) \, \p_{\tilde x} j(\tilde x) \tilde y j(\tilde y) \, d\tilde x d\tilde y  \nonumber \\
& \quad + \int_{\R^2} \big ( (y-\epsilon_n\tilde y) \big (\p_x Q(x-\epsilon_n\tilde x) - \p_x Q(x) \big ) \p_y h_n (x-\epsilon_n\tilde x, y-\epsilon_n\tilde y) \, j(\tilde x) j(\tilde y) \, d\tilde x d\tilde y\nonumber \\
& \quad - \epsilon_n \int_{\R^2} \tilde y \p_x Q(x) \p_y h_n (x-\epsilon_n \tilde x, y-\epsilon_n\tilde y) \, j(\tilde x) j(\tilde y) \, d\tilde x d\tilde y.\nonumber 
\end{align}
Here we have used the fact that $\int_\R \tilde y j(\tilde y)d\tilde y=0$ since $j$ can be chosen as an even function. Hence we have 
\begin{align*}
& \| y(\p_x- \p_x Q (x)\, \p_y ) (h_n -  j_{\epsilon_n} *_x  j_{\epsilon_n} *_y  h_n)\|_{L^2(\R^2_+)} \\
& \leq \int_{\R^2} \| Sh_n - Sh_n (\cdot-\epsilon_n \tilde x, \cdot-\epsilon_n \tilde y)\|_{L^2(\R^2_+)}  j(\tilde x) j(\tilde y) \, d\tilde x d\tilde y  \nonumber \\
& \quad + C\int_{|\tilde y|\leq 2} \| h_n (\cdot, \cdot-\epsilon_n \tilde y)- h_n \|_{L^2}  \, d\tilde y\\
& \quad  + C m(n) \int_\R \| \p_x Q(\cdot-\epsilon_n \tilde x) - \p_x Q\|_{L^3_x} j(\tilde x) \, d\tilde x \| \p_y h_n \|_{L^6_xL^2_y} \\
& \quad + \epsilon_n \| \p_x Q\|_{L^3_x} \| \p_y h_n \|_{L^6_xL^2_y}\\
& \leq \int_{\R^2} \| Sh_n - Sh_n (\cdot-\epsilon_n \tilde x, \cdot-\epsilon_n \tilde y)\|_{L^2}  j(\tilde x) j(\tilde y) \, d\tilde x d\tilde y  \nonumber \\
& \quad + C\epsilon_n  \| \p_y h_n\|_{L^2}\\
& \quad  + C m(n) \int_\R \| \p_x Q(\cdot-\epsilon_n \tilde x) - \p_x Q\|_{L^3_x} j(\tilde x) \, d\tilde x  \big (\| \p_y \tilde f \|_{L^6_xL^2_y}  + m(n)^{-1} n^\frac{1}{12} \| \tilde f \|_{L^{12}_xL^2_y} \big )\\
& \quad + \epsilon_n \| \p_x Q\|_{L^3_x} (\| \p_y \tilde f \|_{L^6_xL^2_y}  + m(n)^{-1} n^\frac{1}{12} \| \tilde f \|_{L^{12}_xL^2_y})\\
& \leq \int_{\R^2} \| S\tilde f   - S \tilde f (\cdot-\epsilon_n \tilde x, \cdot-\epsilon_n \tilde y)\|_{L^2(\R^2_+)}  j(\tilde x) j(\tilde y) \, d\tilde x d\tilde y  \nonumber \\
& \quad  + C m(n)n^{-1}n^\frac{5}{12}  \| \tilde f \|_{L^{12}_x L^2_y}+ C \|\p_x Q\|_{L^\frac{12}{5}_x} \| \tilde f \|_{L^{12}_x L^2_y (\{m(n) \leq |y|\leq 2 m(n)\})}\\ 
& \quad + C\epsilon_n  \big ( m(n)^{-1} n^\frac{5}{12} \| \tilde f \|_{L^{12}_x L^2_y} + n^\frac13 \| \p_y \tilde f \|_{L^6_xL^2_y} \big )\\
& \quad  + C m(n) \int_\R \| \p_x Q(\cdot-\epsilon_n \tilde x) - \p_x Q\|_{L^3_x} j(\tilde x) \, d\tilde x  \big (\| \p_y \tilde f \|_{L^6_xL^2_y}  + \| \tilde f\|_{L^{12}_xL^2_y} \big )\\
& \quad + \epsilon_n \| \p_x Q\|_{L^3_x} (\| \p_y \tilde f \|_{L^6_xL^2_y}  + m(n)^{-1} n^\frac{1}{12} \| \tilde f \|_{L^{12}_xL^2_y})\\
& \leq  \int_{\R^2} \| S\tilde f  - S \tilde f (\cdot-\epsilon_n \tilde x, \cdot-\epsilon_n \tilde y)\|_{L^2(\R^2_+)}  j(\tilde x) j(\tilde y) \, d\tilde x d\tilde y  \nonumber \\
& \quad  + Cn^{-\frac16} \| \tilde f\|_{L^6_xL^2_y} + C \|\p_x Q\|_{L^\frac{12}{5}_x}\| \tilde f \|_{L^{12}_x L^2_y (\{m(n) \leq |y|\leq 2 m(n)\})}\\ 
& \quad + C\epsilon_n  n^\frac13  (1+\|\p_x Q\|_{L^3_x}) \big (  \| \tilde f \|_{L^{12}_x L^2_y} +  \| \p_y \tilde f\|_{L^6_xL^2_y} \big )\\
& \quad  + C n^\frac{5}{12} \int_\R \| \p_x Q(\cdot-\epsilon_n \tilde x) - \p_x Q\|_{L^3_x} j(\tilde x) \, d\tilde x  \big (\| \p_y \tilde f \|_{L^6_xL^2_y}  + \| \tilde f\|_{L^{12}_xL^2_y} \big ).
\end{align*}
Now we take $\epsilon_n$ so that 
\begin{align}\label{ep_n}
\epsilon_n n^\frac13 + n^\frac{5}{12}  \int_\R \| \p_x Q(\cdot-\epsilon_n \tilde x) - \p_x Q\|_{L^3_x} j(\tilde x) \, d\tilde x \leq \frac1n.
\end{align}
This gives 
\begin{align*}
& \| y(\p_x- \p_x Q (x)\, \p_y ) (h_n -  j_{\epsilon_n} *_x  j_{\epsilon_n} *_y  h_n)\|_{L^2} \\
&\leq  \int_{\R^2} \| S\tilde f  - S \tilde f (\cdot-\epsilon_n \tilde x, \cdot-\epsilon_n \tilde y)\|_{L^2(\R^2_+)}  j(\tilde x) j(\tilde y) \, d\tilde x d\tilde y  \nonumber \\
& \quad  + Cn^{-\frac16} \| \tilde f\|_{L^6_xL^2_y} + C \|\p_x Q\|_{L^\frac{12}{5}_x}\| \tilde f \|_{L^{12}_x L^2_y (\{m(n) \leq |y|\leq 2 m(n)\})}\\ 
& \quad + Cn^{-1}  (1+\|\p_x Q\|_{L^3_x}) \big (  \| \tilde f \|_{L^{12}_x L^2_y} +  \| \p_y \tilde f\|_{L^6_xL^2_y} \big ).
\end{align*}
It suffices to consider the first term of the right-hand side. Since $S\tilde f\in L^2_x (\R; L^2_y (\{y>-4\}))$, by introducing the cut-off $\varphi(y)=1$ for $y\geq -1$ and $\varphi (y) =0$ for $y\leq -2$, we have $\varphi S\tilde f\in L^2(\R^2)$ and
\begin{align*}
& \int_{\R^2} \| S\tilde f  - S \tilde f (\cdot-\epsilon_n \tilde x, \cdot-\epsilon_n \tilde y)\|_{L^2(\R^2_+)}  j(\tilde x) j(\tilde y) \, d\tilde x d\tilde y \\
 & = \int_{\R^2} \| \varphi S\tilde f  - (\varphi S \tilde f) (\cdot-\epsilon_n \tilde x, \cdot-\epsilon_n \tilde y) \|_{L^2(\R^2_+)}  j(\tilde x) j(\tilde y) \, d\tilde x d\tilde y\\
 & \leq  \int_{\R^2} \| \varphi S\tilde f  - (\varphi S \tilde f) (\cdot-\epsilon_n \tilde x, \cdot-\epsilon_n \tilde y) \|_{L^2} \,  j(\tilde x) j(\tilde y) \, d\tilde x d\tilde  y\\
 & \rightarrow 0 \quad (n\rightarrow \infty).
\end{align*}
The proof is complete.

\section{Proof of Proposition \ref{prop.w}}

For simplicity we write $\Phi$ for $\Phi (x,y)$. We see 
\begin{align*}
\p_y w & = (1-A\p_y \chi_R)(\p_y \omega) (\cdot,\Phi), \\
\begin{split}
\p_x w & = (\p_x \omega)(\cdot,\Phi) -\chi_R  \p_x A \,(\p_y \omega)(\cdot,\Phi) \\
& =(\p_x \omega)(\cdot,\Phi) - \p_x A \,(\p_y \omega)(\cdot,\Phi) +  (1-A\p_y \chi_R)^{-1}(1-\chi_R) \p_x A\, \p_y w,
\end{split}
\end{align*}
and 
\begin{align*}
\p_y^2 w & = (1-A\p_y\chi_R)^2 (\p_y^2 \omega) (\cdot,\Phi) - A\p_y^2\chi_R (\p_y\omega) (\cdot,\Phi)\\
& = (\p_y^2 \omega) (\cdot,\Phi)  + \big (-2A\p_y\chi_R+(A\p_y\chi_R)^2\big ) (\p_y^2\omega) (\cdot,\Phi) -A \p_y^2\chi_R (\p_y\omega) (\cdot,\Phi)\\
& = \p_y^2 \omega (\cdot,\Phi)  + (1-A\p_y\chi_R)^{-2} \big (-2A\p_y\chi_R+(A\p_y\chi_R)^2\big ) \p_y^2w \\
& \quad + (1-A\p_y\chi_R)^{-3} \big (-2A\p_y\chi_R+(A\p_y\chi_R)^2\big ) A \p_y^2\chi_R   \p_y w  \\
& \quad -  (1-A\p_y \chi_R)^{-1} A \p_y^2\chi_R \, \p_y w\\
& = \p_y^2 \omega (\cdot,\Phi)  +\theta_3 \p_y^2 w + \theta_{2,3}\p_y w .
\end{align*}
Let us recall that $\omega$ satisfies 
\begin{align*}
y\p_x\omega-\p_y^2\omega = -u\p_x\omega - \big ( - y\p_x A + \int_0^y \int_z^\infty (\p_x\omega - \p_x A\, \p_{\tilde z} \omega ) d\tilde zd z - \p_x A\, u\big ) \p_y \omega,
\end{align*}
where the identity \eqref{eq.v} for $v$ is used. Hence,
\begin{align}
y (\p_x\omega -\p_x A\,\p_y\omega) -\p_y^2 \omega = -u (\p_x\omega -\p_x A\,\p_y \omega) - \int_0^y \int_z^\infty (\p_x\omega - \p_x A\, \p_{\tilde z}\omega ) d\tilde zd z  \, \p_y \omega.
\end{align}
Thus, we have from $y=\Phi + A\chi_R$,
\begin{align*}
y\p_xw-\p_y^2 w & = \Phi \big ( (\p_x \omega)(\cdot,\Phi) - \p_x A \,(\p_y \omega)(\cdot,\Phi)\big ) -(\p_y^2 \omega )(\cdot,\Phi)\\
& \quad + A\chi_R \big ( (\p_x \omega)(\cdot,\Phi)-\p_x A \,(\p_y \omega)(\cdot,\Phi)\big )  +  y (1-A\p_y \chi_R)^{-1}(1-\chi_R) \p_x A\, \p_y w\\
& \quad  -\theta_3  \p_y^2w -\theta_{2,3} \p_y w\\
& = -u (\cdot,\Phi) \big ( (\p_x\omega) (\cdot,\Phi) - \p_x A\, (\p_y\omega) (\cdot,\Phi)\big ) - \theta_{2,1} \p_y w\\
& \quad + A\chi_R \big ( (\p_x \omega)(\cdot,\Phi) - \p_x A \,(\p_y \omega)(\cdot,\Phi)\big ) +  y (1-A\p_y \chi_R)^{-1}(1-\chi_R) \p_x A\, \p_y w\\
& \quad -\theta_3  \p_y^2w -\theta_{2,3} \p_y w\\
& = -\theta_1 \p_x w - (u(\cdot,\Phi) - \Phi)   (1-A\p_y \chi_R)^{-1}(1-\chi_R) \p_x A\, \p_y w  \\
& \quad -\theta_{2,1}\p_yw -\theta_3  \p_y^2w -\theta_{2,3} \p_y w\\
& = -\theta_1\p_xw -\sum_{j=1,2,3}\theta_{2,j} \p_y w -\theta_3\p_y^2 w.
\end{align*}
Since $w=\omega$ near $y=0$ it is clear that $\p_yw|_{y=0}=\p_y\omega|_{y=0}$. 
Finally, we have from $I_\infty[\omega]=A$ that $\displaystyle \int_0^\infty \omega (x,\Phi(x,y) ) \p_y \Phi (x,y) \, d y = A (x)$, that is, $I_\infty [w (1-A\p_y \chi_R)] =A$. The proof is complete.

\vspace{5 mm}

\noindent \textbf{Acknowledgements:} SI gratefully acknowledges support by NSF award DMS2306528, a UC Davis Society of Hellman Fellowship award, and the hospitality of Kyoto University over the 2023 KTGU Mathematics Workshop during which this work was initiated. YM acknowledges the support of JSPS KAKENHI Grant Number 20K03698, 19H05597, 20H00118, 21H00991, 21H04433, 23H01082, 23K25779.

\end{document}